\documentclass[11pt,reqno,tbtags]{amsart}
\allowdisplaybreaks

\usepackage{amsthm,amssymb,amsfonts,amsmath}
\usepackage{mathtools}
\usepackage{amscd} 
\usepackage{mathrsfs}
\usepackage[numbers, sort&compress]{natbib}
\usepackage{subfigure,graphicx}
\usepackage{vmargin, enumerate}
\usepackage{setspace}
\usepackage{paralist}
\usepackage[pagebackref=true]{hyperref}
\usepackage{color}
\usepackage{stmaryrd} 
\usepackage[refpage,intoc]{nomencl} 
\usepackage{tikz}
\usetikzlibrary{arrows}
\usetikzlibrary{decorations.markings}
\usepackage{wrapfig}
\usepackage{mathabx}
\usepackage{etoolbox}
\usepackage{caption}
\usepackage{fancybox}
\usepackage{bbm}
\usepackage{mathrsfs}
\usepackage{textcase}
\usepackage{float}
\usepackage{type1cm}

\numberwithin{equation}{section}

\newcommand{\R}{\mathbb{R}}

\newcommand{\N}{{\mathbb N}}

\newcommand{\K}{\mathbb{K}}

\newcommand{\E}[1]{{\mathbb E}\left[#1\right]}

\newcommand{\p}[1]{{\mathbb P}\left(#1\right)}

\newcommand{\I}[1]{{\mathbf 1}_{[#1]}}

\newcommand{\instep}{\hspace*{0.3cm}}

\newtheorem{theorem}{Theorem}[section]
\newtheorem{lem}[theorem]{Lemma}

\newtheorem{prop}[theorem]{Proposition}
\newtheorem{cor}[theorem]{Corollary}
\newtheorem{fact}[theorem]{Fact}

\newtheorem{rmk}[theorem]{Remark}

\newcommand\cB{\mathcal B}

\newcommand\cF{\mathcal F}

\newcommand\cM{\mathcal M}

\newcommand\cP{\mathcal P}

\newcommand\cS{{\mathcal S}}
\newcommand\cT{{\mathcal T}}
\newcommand\cU{{\mathcal U}}

\newcommand{\rB}{\mathrm{B}} 
 
\newcommand{\rD}{\mathrm{D}}

\newcommand{\rH}{\mathrm{H}}

\newcommand{\rP}{\mathrm{P}}

\newcommand{\rT}{\mathrm{T}}

\newcommand{\rZ}{\mathrm{Z}}

\newcommand{\rd}{\mathrm{d}} 

\newcommand{\bT}{\mathbf{T}} 
 
\newcommand{\bV}{\mathbf{V}}

\newcommand{\bn}{\mathbf{n}}

\newcommand{\be}{\mathbf{e}}

\newcommand{\refT}[1]{Theorem~\ref{#1}}
\newcommand{\refC}[1]{Corollary~\ref{#1}}

\newcommand{\refL}[1]{Lemma~\ref{#1}}

\newcommand{\refS}[1]{Section~\ref{#1}}
\newcommand{\refP}[1]{Proposition~\ref{#1}}

\newcommand{\refFt}[1]{Fact~\ref{#1}}

\newcommand{\pran}[1]{\left(#1\right)}

\providecommand{\eps}{}
\renewcommand{\eps}{\varepsilon}
\providecommand{\veps}{}
\renewcommand{\veps}{\varepsilon}
\providecommand{\ora}[1]{}
\renewcommand{\ora}[1]{\overrightarrow{#1}}

\newcommand{\dgr}{\ensuremath{\rd_{\mathrm{gr}}}}
\newcommand{\dleb}{\ensuremath{\rd_{\mathrm{len}}}}

\newcommand{\dghp}{\ensuremath{\rd_{\mathrm{GHP}}}}
\newcommand{\dgh}{\ensuremath{\rd_{\mathrm{GH}}}}

\newcommand{\dis}{\mbox{dis}}

\newcommand{\eqdist}{\ensuremath{\stackrel{\mathrm{d}}{=}}}

\renewcommand{\leq}{\leqslant}
\renewcommand{\geq}{\geqslant}
\renewcommand{\phi}{\varphi}

\newcommand{\eq}{\eqref}
\newcommand{\bigo}{\mathrm{O}}

\newcommand{\Bi}{\mathop{\mathrm{Bi}}}

\newcommand{\IE}{\mathbbm{E}}
\newcommand{\IP}{\mathbbm{P}}

\newcommand{\law}{\mathscr{L}}

\def\be#1{\begin{equation*}#1\end{equation*}}
\def\ben#1{\begin{equation}#1\end{equation}}

\def\ba#1{\begin{align*}#1\end{align*}}

\def\bklr#1{\bigl(#1\bigr)}
\def\bbklr#1{\Bigl(#1\Bigr)}

\def\bkle#1{\bigl[#1\bigr]}

\def\bklg#1{\bigl\{#1\bigr\}}

\def\abs#1{\vert#1\vert}

\def\floor#1{{\lfloor#1\rfloor}}

\def\B{{\mathrm{Beta}}}

\def\bPclas#1#2#3{{\mathcal{P}\bklr{{\textstyle{#1\atop #2}};#3}}}

\def\bPimm#1#2#3#4{{\mathcal{P}_{\mathrm{Im}}^#1\bklr{{\textstyle{#2\atop #3}};#4}}}
\def\bbPimm#1#2#3#4{{\mathcal{P}_{\mathrm{Im}}^#1\bbklr{{\textstyle{#2\atop #3\strut}};#4}}}

\hypersetup{
    bookmarks=true,         
    unicode=false,          
    pdftoolbar=true,        
    pdfmenubar=true,        
    pdffitwindow=true,      
    pdftitle={},    
    pdfauthor={},     
    pdfsubject={Mathematics},   
    pdfnewwindow=true,      
    pdfkeywords={keywords}, 
    colorlinks=true,       
    linkcolor=blue,          
    citecolor=blue,        
    filecolor=blue,      
    urlcolor=blue           
}

\begingroup
  \divide\count255 by 60
  \multiply\count255 by -60
  \advance\count255 by \time
  \ifnum \count255 < 10 \xdef\oclock{\the\count1:0\the\count255}
  \else\xdef\oclock{\the\count1:\the\count255}\fi
\endgroup

\DeclareRobustCommand{\SkipTocEntry}[5]{}
\usepackage[foot]{amsaddr}

\begin{document}

\title{Scaling limits for some random trees constructed inhomogeneously}
\author{Nathan Ross and Yuting Wen}
\address{School of Mathematics and Statistics, University of Melbourne, Parkville, VIC, Australia}
\date{\today{}}
\maketitle

\begin{abstract}
We define some new sequences of recursively constructed random combinatorial trees,
and show that,  after properly rescaling graph distance and equipping the trees with the uniform measure on vertices, each sequence converges almost surely to a real tree in the Gromov-Hausdorff-Prokhorov sense.
The limiting real trees are constructed via line-breaking the real half-line with a Poisson process having rate $(\ell+1)t^\ell dt$, for each positive integer $\ell$, and the growth of the combinatorial trees may be viewed as an inhomogeneous generalization of R\'emy's algorithm.
\end{abstract}


\section{Introduction}
\label{sec:intro}

Understanding the structure of large random trees and graphs
is an important topic of much recent interest in mathematics, statistics, and science. Random trees appear in population genetics and computer science, and statistical data with network structure is now generated in many fields. One important approach to studying a large random discrete structure is to determine limiting behavior as its size tends to infinity, in particular the structure may converge in a suitable sense to a limit object. 
Two well-known illustrations of this approach are the classical functional central limit theorem and the recently developed notion of dense graph limits (so-called graphons). In this paper we are interested in a third setting that has been an important and active research area for the last 25 years: continuum tree limits of combinatorial (i.e., graph-theoretic) trees; here trees are viewed as measured metric spaces and convergence is in the Gromov-Hausdorff-Prokhorov (GHP) topology.
Necessary background on GHP topology  is provided in Section~\ref{sec:gh},
but roughly speaking, two measured metric spaces are close in the GHP topology if each can be isometrically embedded into a common metric space so that both of their Hausdorff distance and the L\'evy-Prokhorov distance between the push-forwards of their measures are small.

To clarify the upcoming discussion, we first mention how trees are viewed as metric spaces; see \citet{E} for a more thorough treatment.  Throughout the article, trees are {\em not} embedded in the plane (i.e., unordered). A compact metric space $(\cT,\dleb)$ is a {\em real tree} if the following two properties hold for every $x,y\in\cT$.
\begin{enumerate}
\item There is a unique isometric map $f_{x,y}$ from $[0,\dleb(x,y)]$ into $\cT$ such that $f_{x,y}(0)=x$ and $f_{x,y}(\dleb(x,y))=y$.
\item If $g$ is a continuous injective map from $[0,1]$ into $\cT$ such that $g(0)=x$ and $g(1)=y$, then we have $g([0,1]) = f_{x,y}([0,\dleb(x,y)])$.
\end{enumerate} 
We call the metric $\dleb$ the {\em intrinsic length} metric on $\cT$. For every $x,y \in\cT$, we call $f_{x,y}$ a (non-graph-theoretic) {\em path} in $\cT$, and denote by $|f_{x,y}| \coloneqq \dleb(x,y)$ the intrinsic length of the path. For ease of notation, write $\cT$, instead of $(\cT,\dleb)$, for a real tree. A {\em leaf} of a real tree $\cT$ is a point $x\in \cT$ such that $\cT\setminus \{x\}$ is connected.

To emphasize the difference from real trees, we call graph-theoretic trees {\em combinatorial trees}. Given a combinatorial tree $T$, let $v(T)$ be the vertex-set of $T$, denote by $\dgr$ the graph distance on $T$, and view $T$ as the metric space $(v(T),\dgr)$. All edges and paths in a combinatorial tree are of graph-theoretic sense (i.e., an edge has length $1$, and the length of a path is the number of edges in it). We often consider {\em rooted trees} which are pairs $(T,u)$, where $T$ is a combinatorial (resp. real) tree, and $u$ is a distinguished vertex (resp. point) of $T$. We call $u$ the root of $(T,u)$.

The fundamental results for tree convergence in our setting are due to
\citet{A1,A2,A3}, who constructed and studied a limit object now called the  Brownian continuum random tree (BCRT). Aldous showed that the BCRT is the 
\begin{enumerate}
\item limit as the number of vertices tends to infinity of certain random combinatorial trees with rescaled edge-lengths (more specifically, the combinatorial trees are those formed from a critical Galton-Watson branching process with finite variance offspring distribution and conditioned on their numbers of vertices),
\item limit of a Poisson line-breaking construction,
\item real tree with contour process equal in distribution to Brownian excursion, and
\item real tree having a certain finite-dimensional distribution on $k$-leaf trees obtained as subtrees spanned by the root and $k$ leaves chosen independently according to a mass measure.
\end{enumerate}

There has been an enormous amount of literature extending, generalizing, and embellishing the results of \cite{A1,A2,A3}. One direction of extension is showing convergence of other families of rescaled combinatorial trees to the BCRT;  see, e.g., \citet{HM, K, MM, Riz}. Another type of extension, and that considered in this paper, is constructing and studying other continuum random trees (CRTs) via some analog of part or all of Items (1-4) above. Well-known examples are the inhomogeneous CRT; see \citet{AP1, AP2, AMP}; the self-similar fragmentation trees; see \citet{HM}; and stable trees; see \citet{D, GH}; see also the references in those papers. 

The general versions of the constructions of Items~(1) and~(2) are most important for this paper. Focusing on Item~(1),
an important class of combinatorial trees that converge to CRTs are those given by various recursive constructions. In these models, a growing sequence of random combinatorial trees $(\rT(n):n\in\N)$ is defined so that $\rT(n+1)$ is constructed conditional on $\rT(n)$ by adding vertices and edges according to specified random rules. Examples of such constructions are R\'emy's algorithm for recursively constructing uniformly chosen leaf-labeled full binary trees; \cite{Remy}; Marchal's generalization of R\'emy's algorithm; \cite{M}; Ford's $\alpha$-model and generalizations; \citet{CFW}; and others: \citet{HS}, \citet{PW09}, \citet{PRW}, \citet{PW}, \citet{RW}. 

The (sometimes Poisson) line-breaking constructions of Item~(2) starts with a sequence of growing random real trees $(\cT_k:k\in\N)$, and then a real tree is defined to be the closure of the union of the sequence. The sequence is recursively constructed: given $\cT_k$, we create $\cT_{k+1}$ by attaching the end of a branch of a random length to a randomly chosen point of~$\cT_k$. To describe the Poisson line-breaking construction of the scaled BCRT, let $C_1, C_2, \ldots$ be the points of an inhomogeneous Poisson process on $(0,\infty)$ with intensity measure $2t dt$. Then we set $\cT_1$ to be a single branch of length $C_1$, and recursively construct $\cT_{k+1}$ from $\cT_k$ by attaching the end of a branch of length $C_{k+1}-C_{k}$ to a uniform point of $\cT_k$. The closure of this sequence is a compact metric space with a measure supported on the leaves that is the weak limit of the uniform measure on the sequence trees.
An important remark for our purposes is that it is possible to embed R\'emy's algorithm into this Poisson line-breaking construction of BCRT, and this embedding can be used to show that uniformly chosen full binary trees with rescaled edge-lengths converge to the BCRT as the number of vertices goes to infinity. Similarly, Marchal's algorithm can be embedded into a line-breaking construction of stable trees, and this embedding can be used to show convergence of Marchal's trees with rescaled edge-lengths to continuum stable trees \cite{GH}.

In this paper, we extend these ideas by defining a new family of sequences of growing recursively constructed combinatorial trees in the spirit of R\'emy's algorithm and show these sequences of trees can be embedded into appropriate Poisson line-breaking constructions. We use this embedding to show that the sequence of combinatorial trees, equipped with the uniform measure on vertices, almost surely converges in the GHP topology to the closure of the union of the line-breaking constructed trees equipped with a probability measure supported on the leaves.  \citet{CH} recently systematically studied the trees that appear as limits, and determined useful properties regarding compactness, boundedness, asymptotic height, and Hausdorff dimension.
See also the recent works of \citet{ADGO} and \citet{H2} for related constructions and discussions.

\subsection{Main result}
For each~$\ell\in\N:=\{1,2,3,\ldots\}$, we define a sequence of growing random combinatorial trees endowed with the uniform probability measure and a real tree limit.
 To ease notation, fix $\ell\in\N$.  

\subsubsection*{Construction of the combinatorial trees}  Consider growing a sequence of random combinatorial trees $(\rT(n):n\in\N)$ in the following inhomogeneous manner. Let $\rT(0)$ be a single (graph-theoretic) edge, call one endpoint a leaf, denoted by $L_1$, and call the other endpoint the root of $\rT(0)$, denoted by $v_0$. For each $n\in\N$, given $\rT(n-1)$, insert a new vertex $v_{n}$ in the interior of a uniformly chosen edge of $\rT(n-1)$. If $\ell$ divides $n$, then, at the same time as $v_n$ appears, insert an edge connecting $v_n$ and a new leaf, denoted by $L_{1+\frac{n}{\ell}}$. The resulting tree $\rT(n)$ is rooted at $v_0$. Note that for $k\in\N$, $\rT(k\ell-1)$ has $k$ leaves and $\rT(k\ell)$ has $k+1$ leaves.
In addition, for all $n\in\N$, let $\nu_n$ be the uniform probability measure over~$v(\rT(n))$. Note that for $\ell\geq 2$ there are degree-2 vertices in the trees and that the case $\ell=1$ coincides with R\'emy's algorithm~\cite{Remy}. 

\subsubsection*{Construction of the limiting real trees} 
The limiting real trees have been recently studied~\citep{CH} and are
generalizations of the line-breaking construction for the BCRT described above and due to \citet{A3}. Given $a=(a_k:k\in\N)\subset \R_+$, we construct a sequence of random real trees $(\cT^a_k:k\in\N)$ by starting with $\cT^a_1$, which is made of a single branch of length $a_1$. For integer $k\ge2$, we recursively construct $\cT^a_k$ from $\cT^a_{k-1}$ by attaching the end of a branch of length $a_{k}$ to a point chosen uniformly from $\cT^a_{k-1}$. For all $k\in\N$, root $\cT^a_k$ at an arbitrarily fixed end of the initial branch. Furthermore, let $\cT^a$ be the closure of $\cT^a_k$ as $k\to\infty$. Next, write $C_{0}=0$, and let $C_1,C_2,\ldots$ be the times in $(0,\infty)$ of an inhomogeneous Poisson process of rate $(\ell+1)t^\ell dt$. For all $k\in\N$, write $\cT_k = \cT^{(C_k-C_{k-1}:k\in\N)}_k$. Finally, let $\cT$ be the completion of $\cT_k$ as $k\to\infty$, which is a random real tree with intrinsic length metric $\dleb$. 
\citet{CH} show that the limit tree~$\cT$ is almost surely compact and has a natural  ``uniform" probability measure supported on the leaves; see Theorems~\ref{thm:ch} and~\ref{thm:meas} below. 
Note that the case $\ell=1$ is exactly the Poisson line-breaking construction of the BCRT.

We can now state our main result.
For any $a>0$, write $a\cdot \dgr$ for the metric so that $(a\cdot \dgr)(x,y) = a\cdot \dgr(x,y)$.
For the  remainder of the paper, define
\[
\alpha=\alpha(\ell) = \frac{\ell}{\ell+1}~\mbox{ and }~ c=c(\ell)=\frac{\ell^\alpha}{\ell+1}.
\]

\begin{theorem}\label{thm:gh}
There is a probability space where we can construct copies of $(\rT(n):n\in\N)$ and $(\cT_k:k\in\N)$ such that the following holds. There almost surely exists a probability measure $\mu$ supported by the leaves of $\cT$ such that
\[
\left(v(\rT(n)),\frac{c}{n^{\alpha}}\cdot \dgr,\nu_n\right) \to (\cT,\dleb,\mu) 
\]
almost surely for the Gromov-Hausdorff-Prokhorov topology as $n\to\infty$.
\end{theorem}

The proof of \refT{thm:gh} follows from three main steps. First, we can embed the combinatorial trees into the Poisson line-breaking trees (Proposition~\ref{prop:dist}). The embedding follows from beta-gamma algebra and is similar in spirit to that described in \citep[Proposition~3.7]{GH} for Marchal's algorithm. Second, we can use the embedding to show that for 
$\rT_k(n)$ defined to be the subtree of $\rT(n)$ spanned by the root and the first $k$ leaves, $\rT_k(n)$ is close to $\cT_k$ even for growing $k$ (Proposition~\ref{prop:finite}). Essentially this requires careful analysis of distances and masses in the combinatorial tree, which in turn boils down to understanding a time inhomogeneous P\'olya urn model studied by \citet{PRR14,PRR}, where distributional convergence results complementary to this paper are derived. 
Note also that once the correspondence to the urn model is made (in Section~\ref{sec:polya}), the choice of the scaling constant $c$ agrees with that of \cite[Proposition~2.1]{PRR14}; in our work,~$c$ is chosen to cancel the leading term in \eq{exponent}. Finally, we show what is left over in $\rT(n)$ outside of $\rT_k(n)$ is sufficiently small (Proposition~\ref{prop:tight}). 
This tightness argument requires careful analysis of two P\'olya urn models and 
an understanding of exchangeable random ``decorated'' masses.

The layout of the remainder of the paper is as follows. We present the three key propositions and a detailed proof outline in the last subsection of this introduction. In Section~\ref{sec:gh}, we provide necessary background on GHP convergence, and then we prove the three propositions in Sections~\ref{sec:coupling},~\ref{sec:as}, and~\ref{sec:tight}. We conclude this subsection of the introduction with a few remarks contextualizing our result and discussing further work.

For related work, as previously discussed, there is much interest in limits of recursively constructed trees. However, typically the models considered have some nice consistency properties such as 
Markov branching (see \cite{H} for a recent review), perhaps with some consistent leaf-labeling, e.g., a regenerative structure as in  \cite{PRW}, or having fully exchangeable leaf-labels. By consideration of small cases, it is clear that the leaf-labeling in our models is not exchangeable and the combinatorial trees do not have the Markov branching property so we cannot directly apply the general theory developed for such models. Also, it is unusual for recursively built  combinatorial tree models of the kind studied here to allow for degree-2 vertices and this case is excluded from some studies. Having GHP convergence results for an example falling outside the general theory is  interesting in its own right, but may also lead to further natural classes of models and general theory. 

There are many avenues for future study. The most obvious open problem is to provide a description analogous to Items~(3) and~(4) above for the limit trees. Moreover, there are many other decompositions and properties of recursively defined trees and their limits that are important and appear in the CRT literature -- what are the analogs of these in our setting? Note that our combinatorial trees provide one path to understanding properties of the limit trees.

\subsection{Proof outline of \refT{thm:gh}}

First, we examine the topologies of the combinatorial trees and the real trees. The idea is to embellish the real trees with random vertices so that the resulting trees, equipped with the graph distance, have the same law as the combinatorial trees.

\subsubsection*{Embellished trees} Write $\cT(0) = \cT_1$. For each $k\in\N$ and $i\in\{1,\ldots,\ell-1\}$, let $\cT((k-1)\ell+i)$ be obtained from inserting a vertex at a random point uniformly chosen with respect to the normalized Lebesgue measure over $\cT((k-1)\ell+i-1)$. Let $\cT'_k$ be formed by inserting a vertex uniformly in $\cT((k-1)\ell+\ell-1)$ and define $\cT(k\ell)$ by attaching a branch of length $C_{k+1}-C_k$ to this last inserted vertex. We call $\cT(1),\cT(2),\ldots$ the {\em embellished trees}, rooted at the same point as $\cT_1$. This construction is analogous to that of the combinatorial trees.

All vertices inserted in the above manner are called the {\em embellished vertices}. For all $k\in\N$ and $i\in\{1,\ldots,\ell-1\}$, if we forget about the embellished vertices, then $\cT((k-1)\ell+i)$ with the intrinsic length metric has the same law as the real tree $\cT_k$. A {\em leaf} of the embellished tree $\cT((k-1)\ell+i)$ is the corresponding leaf of $\cT_k$. A {\em vertex} of the embellished tree $\cT(n)$ is either an embellished vertex, a leaf, or the root. Denote by $v(\cT(n))$ the set of vertices of $\cT(n)$. We view $\cT(n)$ as the union of the (non-graph-theoretic) branches and the vertices, i.e., a hybrid of the real tree and the combinatorial tree.

For all integers $n,k$ with $n\ge (k-1)\ell$, let $\cT_k(n)$ be the subtree of the embellished tree $\cT(n)$ spanned by the root and the first $k$ leaves (in the order of appearance). Analogously, write $\rT_k(n)$ for the subtree of $\rT(n)$ spanned by the root and the first $k$ leaves.

The embellished trees give a way to couple $\rT(n)$ and $\cT(n)$ as follows. Recall that we often write $\rT_k(n) = \big( v(\rT_k(n)),\dgr\big)$ and $\cT_k  = \big(\cT_k,\dleb\big) $.

\begin{prop}\label{prop:dist}
There is a probability space where we can construct copies of $(\rT(n):n\in\N)$, $(\cT(n):n\in\N)$, and $(\cT_k:k\in\N)$ such that
\[
\big(v(\cT_k(n)),\dgr\big)= \rT_k(n)
~\mbox{ and }~
\big(\cT_k(n),\dleb\big) = \cT_k,
\]
for all integers $k,n$ with $n\ge (k-1)\ell$, equalities considered up to isometry-equivalence.
\end{prop}

The proof of \refP{prop:dist}, given in \refS{sec:coupling}, relies on that when vertices are inserted into the embellished tree, branches are fragmented into Dirichlet-distributed lengths. 

\refP{prop:dist} gives us a direct coupling to compare the rescaled graph-theoretic path-lengths of $\rT_k(n)$ and the corresponding intrinsic path-lengths of $\cT_k$, which  leads to our next result showing that the combinatorial trees spanned by a subset of leaves and the analogous subtree of the limit tree are close.

Before stating the result, we need some facts and notation.
Firstly, as discussed in greater detail just below and in Section~\ref{sec:gh},
all the metric spaces appearing in this paper are compact, and so in fact we can define a distance on such metric spaces (modulo isometry-equivalence), denoted~$\dghp$, which induces the GHP topology.
Next, for all $k\in\N$, let $\mu_k$ be the normalized Lebesgue length measure on $\cT_k$. For all integers $k,n$ with $n\ge(k-1)\ell$, write $\nu_{k,n}$ for the uniform probability measure over $v(\rT_k(n))$. For two sequences $g(n), f(n)$, write $g(n) = \Omega(f(n))$ if there exists $C>0$ such that $g(n)\ge C f(n)$ for all $n$, and $g(n) = o(f(n))$ if $g(n)/f(n)\to0$, as $n\to\infty$. 
\begin{prop}\label{prop:finite}
Suppose $k:\N\to\N$ satisfies $\Omega\left((\log n)^{10}\right)$ and $k(n)=o\left(n^{1/10}\right)$. In the probability space where the equalities of \refP{prop:dist} hold, almost surely as $n\to\infty$,
\[
\dghp\left( \bklr{v\bklr{\rT_{k(n)}(n)},\frac{c}{n^\alpha}\cdot \dgr, \nu_{k(n),n}},\big(\cT_{k(n)}, \dleb,\mu_{k(n)}\big)\right) \to 0.
\] 
\end{prop}

The proof of \refP{prop:finite} is given in \refS{sec:as} and relies on  a concentration result (Lemma~\ref{lem:convas} and Corollary~\ref{cor:embellish0}), which says that the number of vertices along a path in $\rT_{k(n)}(n)$ has order $c^{-1}n^\alpha$ times the Lebesgue length of the path, and that the vertices are regularly distributed.

Next, to ensure that $\rT_k(n)$ is close to $\rT(n)$, we need a tightness property of the sequence $(\rT_k(n):k\in\N, n\ge k\ell)$, i.e., the Hausdorff distance between $\rT_{k(n)}(n)$ and $\rT(n)$ is diminishing, and the L\'evy-Prokhorov distance between their uniform probability measures also vanishes in the limit. Recall that $\nu_n$ is the uniform probability measure over $v(\rT(n))$.

\begin{prop}\label{prop:tight}
Suppose $k:\N\to\N$ satisfies $k(n)=\Omega\left(n^{1/100}\right)$ and $k(n)=o\left(n^{1/3}\right)$ and assume now $\ell\geq 2$. Then, almost surely as $n\to\infty$,
\[
\dghp\left( \big( v(\rT_{k(n)}(n)),\frac{c}{n^\alpha}\cdot\dgr, \nu_{k(n),n}\big), \big( v(\rT(n)),\frac{c}{n^\alpha}\cdot\dgr, \nu_{n}\big) \right)
\to 0.
\]
\end{prop}
Note the restriction in Proposition~\ref{prop:tight} to $\ell\geq 2$, which stems from Lemma~\ref{lem:prok} and in particular the proof of Lemma~\ref{lem:maxsize}.  The restriction is due to balancing asymptotic terms and probably some version of the proposition and these lemmas hold for $\ell=1$, but convergence in this case is well-covered in the literature and so it is enough for us to consider $\ell\geq 2$. All other lemmas and propositions in the paper hold for $\ell=1$. 

To establish \refP{prop:tight}, we deduce a height-bound for the subtrees of $\rT(n)$ pendant to $\rT_{k(n)}(n)$, and we also show that subtrees pendant to $\rT_{k(n)}(n)$ are ``uniformly asymptotically negligible'' (a similar property is used in \citet{ABW,Wen}).
That is, Lemmas~\ref{lem:expbd} and~\ref{lem:sum} imply that with $1-o(1)$ probability, the maximal height of the subtrees of $\rT(n)$ pendant to $\rT_{k(n)}(n)$ has order $o(n^\alpha)$ (yielding GH convergence) 
and Lemma~\ref{lem:maxsize} implies that the maximal size of the subtrees has order $o\left(n\cdot k(n)^{-8/(3(\ell+1))}\right)$. By projecting the masses of pendant subtrees onto $\rT_{k(n)}(n)$, we can deduce a bound on the relevant L\'evy-Prokhorov distance.
Details are given in \refS{sec:tight}. 

As shown in the next several results of \cite{CH}, $\cT$ is almost surely compact, which allows for the convergence to hold in the GHP topology instead of, say, the local GHP topology.

\begin{theorem}{\em (\citep[Theorem 1]{CH}).}\label{thm:ch}
Suppose that there exists $\alpha'\in(0,1]$ such that for $a\coloneqq (a_k:k\in\N)\subset \R_+$ we have $a_k\le k^{-\alpha'+o(1)}$ and $\sum_{i=1}^k a_i = k^{1-\alpha'+o(1)}$ as $k\to\infty$. Then $\cT^a$ is almost surely a compact real tree.
\end{theorem}

\begin{fact}{\em (\cite{CH}).}\label{fact:ch}
If $a_k=C_k-C_{k-1}$ for $n\in\N$, then almost surely $a\coloneqq(a_k:k\in\N)$ satisfies the assumption in \refT{thm:ch} for $\alpha'\coloneqq \frac{\ell}{\ell+1}$.
\end{fact}

\begin{theorem}\label{thm:meas}
{\em (\citep[Theorem 4]{CH}).} Almost surely, there exists a probability measure $\mu$ supported by the leaves of $\cT$ such that $\mu_k \to\mu$ weakly as $k\to\infty$.
\end{theorem}

With these results we can now prove Theorem~\ref{thm:gh}.
\begin{proof}[{\bf Proof of \refT{thm:gh}}]
The case $\ell=1$ is just the well-known almost sure GHP convergence of uniform ordered binary trees with uniform measure to the BCRT (e.g, \citet[Theorem~5]{CH13}), so we assume $\ell\geq2$. 
We work on the probability space where the equalities of \refP{prop:dist} hold, and condition on the a.s.\ event that $\cT$ is compact and $\mu$ exists, where $\mu$ is the uniform probability measure supported by the leaves of $\cT$.

Let $k:\N\to\N$ be such that $k(n)= \Omega\left(n^{1/100}\right)$ and $k(n) = o\left(n^{1/10}\right)$. For all $n\in\N$, write $\widehat{\bT}(n)= \left(v(\rT(n)),\frac{c}{n^{\alpha}}\cdot\dgr,\nu_n\right)$, $\widehat{\bT}_{k(n)} = \left(v(\rT_{k(n)}(n)),\frac{c}{n^\alpha}\cdot \dgr, \nu_{k(n),n}\right)$, $\boldsymbol{ \mathcal{T}}_{k(n)} = \left(\cT_{k(n)},\dleb,\mu_{k(n)}\right)$, and $\boldsymbol{\mathcal{T}} = \left(\cT,\dleb,\mu\right)$. Note that
\begin{align}
\dghp\left(\widehat{\bT}(n), \boldsymbol{\cT} \right)\le \dghp\left(\widehat{\bT}(n),\widehat{\bT}_{k(n)}\right)
+\dghp\left(\widehat{\bT}_{k(n)},\boldsymbol{\cT}_{k(n)}\right)
+\dghp\left(\boldsymbol{\cT}_{k(n)},\boldsymbol{\cT}\right).\label{in:ghp1}
\end{align}
By Propositions \ref{prop:finite} and \ref{prop:tight}, a.s. as $n\to\infty$,
\begin{equation}\label{in:ghp2}
\dghp\left(\widehat{\bT}(n),\widehat{\bT}_{k(n)}\right)
+\dghp\left(\widehat{\bT}_{k(n)},\boldsymbol{\cT}_{k(n)}\right)\to0.
\end{equation}

Furthermore, since $\cT_k\coloneqq (\cT_k,\dleb)$ is separable, the weak convergence of measures, i.e., $\mu_k\to\mu$ (a.s.\ exists by \refT{thm:meas}), is equivalent to the convergence of measures in the L\'evy-Prokhorov metric, i.e., $d_\rP(\mu_k,\mu)\to0$, where $d_\rP$ denotes the L\'evy-Prokhorov distance (defined in \refS{sec:gh}) on $\cT\coloneqq \overline{\bigcup_{k\in\N}\cT_k}$ (a.s. compact by \refT{thm:ch}), and $\mu_k$ is viewed as the measure on $\cT$ such that $\mu_k(\cT\setminus \cT_k)=0$. It follows that a.s.\ $\dghp\left(\boldsymbol{\cT}_{k(n)},\boldsymbol{\cT}\right)\to0$. Together with (\ref{in:ghp1}) and (\ref{in:ghp2}), this completes the proof.
\end{proof}

\section{Gromov-Hausdorff-Prokhorov topology}
\label{sec:gh}

In this section we review the definition of GHP distance and the topology it induces, referring the reader to the papers by \citet[Section 6.2]{Mi09} and \citet[Section 2.1]{ABGM} for greater details and further references.

We first give the standard and intuitive definition of GHP distance.
A {\em measured metric space} is a triple $(V,d,\nu)$ where $(V,d)$ is a metric space and $\nu$ is a finite non-negative Borel measure on $V$. 
Let $\rZ\coloneqq (Z,\delta)$ be a metric space. Given non-empty $A\subset Z$ and $\veps>0$, the {\em $\veps$-neighborhood} of $A$ is $A^\veps \coloneqq A_\delta^\veps \coloneqq \left\{x\in Z: \exists y\in A, \delta(x,y)<\veps\right\}$. The {\em Hausdorff distance} $\delta_\rH$ between two non-empty subsets $X,Y$ of $\rZ$ is
\[
\delta_\rH(X,Y) = \inf\left(\veps>0:X\subset Y^\veps,Y\subset X^\veps\right)~.
\]
Next, denote by $\cP(\rZ)$ the collection of all finite non-negative Borel measures on the measurable space $(Z,\cB(\rZ))$, where $\cB(\rZ)$ denotes the Borel $\sigma$-algebra of $\rZ$. The {\em L\'evy-Prokhorov distance} $\delta_\rP : \cP(Z)^2 \to [0,\infty)$ between two measures $\nu$ and $\nu'$ on $Z$ is
\[
\delta_\rP(\nu,\nu') = \inf\left\{\veps>0: \nu(A)\le \nu'(A^\veps) + \veps \mbox{ and } \nu'(A)\le \nu(A^\veps)+\veps,\forall A\in \cB(\rZ)\right\}.
\]
We can now define the standard metric used to define the  Gromov-Hausdorff-Prokhorov topology.
For two measured metric spaces $\bV = (V,d,\nu)$ and $\bV'=(V',d',\nu')$, define
\[
\dghp^\circ(\bV,\bV') = \inf \max\left\{\delta_\rH(\phi(V),\phi'(V')),\delta_\rP(\phi_*\nu,\phi'_*\nu')\right\},
\]
where the infimum is over all metric space $\rZ$ and all isometries $\phi,\phi'$ from $\bV,\bV'$ into $\rZ$,
and where $\phi_*\nu$ and $\phi'_*\nu'$ denote push-forward measures. On the space of measured metric spaces modulo isometry-equivalence (measured metric spaces $(V,d,\nu)$ and $(V',d',\nu')$ are {\em isometry-equivalent} if there exists a measurable bijective isometry $\Phi:V\to V'$ such that $\Phi_* \nu=\nu'$),
$\dghp^\circ$ is a metric that induces the GHP topology.

The definition above can be difficult to use, so we now state 
some alternative notions and results for showing GHP convergence.
For $(V,d)$ and $(V',d')$ two metric spaces, a {\em correspondence} between $V$ and $V'$ is a set $R\subset V\times V'$ such that for every $x\in V$, there is $x'\in V'$ with $(x,x')\in R$, and vice versa. We write $R(V,V')$ for the set of correspondences between $V$ and $V'$. The {\em distortion} of any $R\in R( V, V')$ with respect to $d$ and $d'$ is
\[
\dis\left(R;d,d'\right) = \sup\left\{|d(x,y) - d'(x',y')|: (x,x')\in R,  (y,y')\in R\right\}~.
\]
Furthermore, let $M(V,V')$ be the set of finite non-negative Borel measures on $V\times V'$. Denote by $p$ and $p'$ the projections from $V\times V'$ to $V$ and $V'$, respectively. Let $\nu$ and $\nu'$ be finite non-negative Borel measures on $(V,d)$ and $(V',d')$, respectively. The {\em discrepancy} of $\pi\in M(V,V')$ with respect to $\nu$ and $\nu'$ is 
\[
\rD\left(\pi;\nu,\nu'\right) = \Vert \nu - p_* \pi \Vert + \Vert \nu' - p'_* \pi \Vert,
\]
where $\|\cdot\|$ denotes the total variation for a signed measure. Given measured metric spaces $\bV=(V,d,\nu)$ and $\bV'=(V',d',\nu')$, we define the {\em Gromov-Hausdorff-Prokhorov distance} by
\[
\dghp(\bV,\bV') = \inf \max\left\{\frac12\cdot\dis(R;d,d'),\rD(\pi;\nu,\nu'),\pi(R^c) \right\},
\]
where the infimum is over all $R\in R(V,V')$ and $\pi\in M(V,V')$. 

Writing $\K$ for the set of all {\em compact} measured metric spaces modulo isometry-equivalence, $(\K,\dghp)$ is a Polish space; see \citet{ADH}. GHP convergence refers to convergence in this space. (It can be shown that $\dghp^\circ$ and $\dghp$ induce the same topology on $\K$.) 

The {\em Gromov-Hausdorff distance} between two metric spaces $(V,d)$ and $(V',d')$ is given by $\dgh\left((V,d),(V',d')\right) =  \inf \frac12\cdot\dis(R;d,d')$, where the infimum is over all $R\in R(V,V')$.
\section{Coupling between combinatorial trees and real trees}\label{sec:coupling}

We prove \refP{prop:dist} in this section, starting by recalling some basic facts about Dirichlet distributions. Let $\mathbf{a}=(a_1,\ldots,a_n)\in \R_+^n$. The {\em Dirichlet distribution} with parameter $\mathbf{a}$, denoted by $\mathrm{Dir}(\mathbf{a})$, has density $f(x_1,\ldots,x_n; \mathbf{a}) = \frac{1}{\rB(\mathbf{a})} \prod_{i=1}^n x_i^{a_i-1}$, for $x_1,\ldots,x_n>0$ with $\sum_{i=1}^n x_i=1$ and $\rB(\mathbf{a}) \coloneqq \frac{\prod_{i=1}^n \Gamma(a_i)}{\Gamma(\sum_{i=1}^n a_i)}$. Let $G_i \sim \mathrm{Gamma}(a_i)$ be independent variables for $i=1,\ldots,n$. It is well-known that $\left( \frac{G_{1}}{\sum_{i=1}^n G_{i}},\ldots,\frac{G_{n}}{\sum_{i=1}^n G_{i}}\right)
\sim \mathrm{Dir}(\mathbf{a})$, and this is independent of $
\sum_{i=1}^n G_{i} \sim \mathrm{Gamma}\left(\sum_{i=1}^n a_i\right)$.

\begin{lem}{\em (\citep[Lemma 2.2]{GH}).}\label{lem:dircond}
Suppose that $(X_1,\ldots,X_n) \sim \mathrm{Dir}(1,\ldots,1)$. Let $J$ have the conditional distribution $\p{J=j~|~X_1,\ldots,X_n}=X_j$. Then $\p{J=j} = \frac{1}{n}$. Furthermore, conditioned on $J=j$, 
$(X_1,\ldots,X_n)\sim \mathrm{Dir}(\underbrace{1,\ldots,1}_{j-1},2,\underbrace{1,\ldots,1}_{n-j})$.
Finally, if $U$ is an independent Uniform$(0,1)$-variable, then, conditioned on $J=j$,
\[
(X_1,\ldots,X_{j-1},U\cdot X_j, (1-U)\cdot X_j,X_{j+1},\ldots,X_n) 
\sim \mathrm{Dir}(\underbrace{1,\ldots,1}_{n+1}).
\]
\end{lem}

\begin{fact}
\label{fact:dirbeta}
Fix $k\in\N$ and let $B\sim \mathrm{Beta}(k,1)$. If $(X_1,\ldots,X_k)\sim \mathrm{Dir}(1,\ldots,1)$, independent of $B$, then $(B\cdot X_1,\ldots, B\cdot X_k,1-B) \sim \mathrm{Dir}(1,\ldots,1)$.
\end{fact}

Recall that the real trees $\cT_1,\cT_2,\ldots$ are constructed by aggregating random intervals of lengths $C_1,C_2-C_1,\ldots$, for which we derive the following representation.

\begin{fact}\label{fact:repr}
Let $E_1,E_2,\ldots$ be independent Exponential$(1)$-variables, and let $C_1 = E_1^{\frac{1}{\ell+1}}$. Let $C_1,B_1,B_2,\ldots$ be independent variables such that $B_k\sim \mathrm{Beta}((\ell+1)k,1)$ for all $k\in\N$. For all integer $k\ge2$, let
\begin{equation}\label{eq:repr}
C_k = \frac{C_1}{B_1\cdots B_{k-1}},
\end{equation}
then $C_k \eqdist (E_1+\cdots +E_k)^{\frac{1}{\ell+1}}$, independent of $B_1,\ldots,B_{k-1}$. Moreover, $(C_k:k\in\N)$ are the points of an inhomogeneous Poisson process on $(0,\infty)$ with intensity $(\ell+1)t^\ell dt$.
\end{fact}

Facts \ref{fact:dirbeta} and \ref{fact:repr} follow from standard calculations, so we omit the proof. 

\begin{rmk}
In the sequel, we use the representation (\ref{eq:repr}) of $C_k$ for all integer $k\ge2$.
\end{rmk}

Next, recall the definition of the embellished tree $\cT(n)$ from \refS{sec:intro}. A (non-graph-theoretic) {\em path}, $S$, in an embellished tree $\cT(n)$ is defined as the corresponding path in the underlying real tree, and the path-length, $|S|$, is equal to the intrinsic length of $S$. An {\em edge} in an embellished tree is a path between two adjacent vertices, and the {\em edge-length} refers to the path-length of the edge. We show that the rescaled edge-lengths of the embellished tree $\cT(n)$ are Dirichlet distributed.

Recall that, for all $k\in\N$, $\cT'_k$ is the embellished tree $\cT(k\ell)$ without the latest branch (i.e., the $(k+1)$:th branch, of length $C_{k+1}-C_{k}$), but it includes the the embellished vertex to which the $(k+1)$:th branch is to be attached. For all $k\in\N$ and $i\in\{0,\ldots,\ell-1\}$, let $E_{k,i}(0),\ldots,E_{k,i}((k-1)(\ell+1)+i)$ be the edge-lengths of the embellished tree $\cT((k-1)\ell+i)$, in the order of appearance. For $i=\ell$, let $E_{k,i}(0),\ldots,E_{k,i}((k-1)(\ell+1)+i)$ be the edge-lengths of $\cT'((k-1)\ell+i)$. Finally, let $E_{k+1,0}((\ell+1)k)=C_{k+1}-C_k$. If two edges appear at the same time, the one closer to the root has a smaller index. 

\begin{lem}
\label{lem:frag}
Fix $k\in\N$ and $i\in\{0,\ldots,\ell\}$. We have
\[
\frac{1}{C_k}\cdot \left(E_{k,i}(0),\ldots,E_{k,i}((k-1)(\ell+1)+i)\right) \sim 
\mathrm{Dir}(1,\ldots,1).
\]
\end{lem}

\begin{proof}
We prove by induction on $i$ and $k$. Let $U\sim \mathrm{Uniform}(0,1)$, independent of everything else. $E_{1,0}(0)\sim \mathrm{Dir}(1)$ is trivial. For $k=1$ and $i=1$, we may assume that $E_{1,1}(0) = C_1U$ and $E_{1,1}(1)=C_1(1-U)$. So $\frac{1}{C_1}\cdot \left(E_{1,1}(0),E_{1,1}(1)\right) \sim \mathrm{Dir}(1,1)$. 

Now suppose that the claim holds for some $k\in\N$ and $i\in\{0,\ldots,\ell-1\}$. We are about to insert a vertex uniformly over $\cT((k-1)\ell+i)$, for the normalized Lebesgue length measure. Let $V\sim \mathrm{Uniform}(0,1)$, independent of everything else. Conditioned on selecting the edge with length $E_{k,i}(j)$ to insert such a vertex, for an appropriate $j$, by \refL{lem:dircond} we have that the $(k-1)(\ell+1)+i+2$ dimensional vector
\begin{align*}
\frac{1}{C_k}\cdot \left(E_{k,i}(0),\ldots,E_{k,i}(j)V,E_{k,i}(j)(1-V),\ldots,E_{k,i}((k-1)(\ell+1)+i)\right)
\sim \mathrm{Dir}(1,\ldots,1).
\end{align*}
The above holds regardless of the choice of $j$, so it also holds without conditioning and the claim follows for $k$ and $i+1$.

Next, we show that the claim holds for $k+1$ and $i=0$ as well. Recall that $\cT(k\ell)$ is obtained from attaching a branch of length $C_{k+1}-C_{k}$ to $\cT'_k$. So
\begin{align*}
\left(E_{k+1,0}(0),\ldots,E_{k+1,0}((\ell+1)k)\right) 
=
\left(E_{k,\ell}(0),\ldots,E_{k,\ell}((k-1)(\ell+1)+\ell),C_{k+1}-C_{k}\right).
\end{align*}
It follows from \refFt{fact:repr} that we may write $C_{k} = B_k C_{k+1}$ for $B_k\sim \mathrm{Beta}((\ell+1)k,1)$, independent of $C_{k+1}$. So $
C_{k+1} - C_k = C_{k+1}(1-B_k)$ where $1-B_k \sim \mathrm{Beta}(1,(\ell+1)k)$. We have proved that the claim holds for $k$ and $i=\ell$: 
\[
\left(\frac{E_{k,\ell}(0)}{C_k},\ldots,\frac{E_{k,\ell}((k-1)(\ell+1)+\ell)}{C_k}\right)\sim \mathrm{Dir}(1,\ldots,1).
\]
 Then by \refFt{fact:dirbeta},
\[
\frac{1}{C_{k+1}}\cdot
 \left(E_{k,\ell}(0),\ldots,E_{k,\ell}((k-1)(\ell+1)+\ell),C_{k+1}(1-B_k)\right)
  \sim \mathrm{Dir}(\underbrace{1,\ldots,1}_{(\ell+1)k+1}).
\]
The lemma follows by induction.
\end{proof}

Recall that $v(\cT(n))$ is the union of the embellished vertices, the leaves, and the root, and recall the definition of the combinatorial tree $\rT(n)$ from \refS{sec:intro}.

\begin{lem}\label{lem:dist}
There exists a probability space where 
\begin{equation}\label{eq:1}
((v(\cT(n)),\dgr):n\in\N) =((v(\rT(n)),\dgr):n\in\N) \eqqcolon (\rT(n):n\in\N),
\end{equation}
considered up to isometry-equivalence.
\end{lem}

\begin{proof}
We prove by induction on $n$. For $n=0$, $(v(\cT(0)),\dgr)= \rT(0)$ (both consist of an edge). Now assume that, for $n\in\N$ such that $\ell$ does not divide $n$, $\left(v(\cT(n-1)),\dgr\right) = \rT(n-1)$. We are about to insert a vertex into $\cT(n-1)$ and $\rT(n-1)$ respectively. It follows from Lemmas \ref{lem:dircond} and \ref{lem:frag} that the new vertex has equal probability to land on any edge of $\cT(n-1)$. This holds true for the insertion into $\rT(n-1)$ as well, by construction. It then follows from the induction hypothesis that $(v(\cT(n)),\dgr)$ and $\rT(n)$ have the same law. We may and shall assume that $(v(\cT(n)),\dgr)=\rT(n)$.

Next, we show that the claim also holds for $n\in\N$ such that $\ell$ divides $n$, assuming that $(v(\cT(n-1)),\dgr)=\rT(n-1)$. After inserting a vertex into both $\cT(n-1)$ and $\rT(n-1)$ as above, we additionally attach a new branch to the last inserted vertex. The resulting trees are $\cT(n)$ and $\rT(n)$. It is easily seen that their laws are the same, and we may view them equal. The lemma then follows by induction.
\end{proof}

This lemma immediately yields \refP{prop:dist}. Recall that $\cT_k(n)$ is the subtree of $\cT(n)$ spanned by the root and the first $k$ leaves.

\begin{proof}[{\bf Proof of \refP{prop:dist}}]
By the constructions above, $\left(\cT_k:k\in\N\right)$ has the same distribution as $ \left((\cT_k((k-1)\ell),\dleb):k\in\N\right)=\left((\cT_k((k-1)\ell+1),\dleb):k\in\N\right)=\ldots$ We may and shall assume that 
\begin{equation}\label{eq:2}
(\cT_k(n),\dleb) = \cT_k 
\end{equation}
for all integers $k,n$ with $n\ge(k-1)\ell$. In the product of the probability spaces where (\ref{eq:1}) and (\ref{eq:2}) hold respectively, the proposition easily follows.
\end{proof}

\section{Almost sure convergence for subtrees with finite leaves}\label{sec:as}

Hereafter, we work in the probability space where the equalities of \refP{prop:dist} hold. In this section we prove \refP{prop:finite}, which requires the following lemmas.

\begin{lem}\label{lem:bd}
For all integer $k\ge (7/3)^4$, $\p{\left\vert C_k^{\ell+1} - k\right\vert \ge k^{3/4}}
\le 2e^{-k^{1/2}/4}$.
\end{lem}

\begin{cor}\label{cor:bd2}
Let $k:\N\to\N$ be such that $k(n)^\alpha = o(n^{\alpha-1/4})$ and $k(n)\to\infty$. Then for sufficiently large $n$, with probability greater than $ 1-2\sum\limits_{m=k(n)}^{\lfloor n/\ell\rfloor } e^{-m^{1/2}/4}$, we have $\left\vert C_{k(n)}-k(n)^{\frac{1}{\ell+1}}\right\vert < 10 k(n)^{\frac{1}{\ell+1}-\frac14}$ and $\left(\frac{n}{\ell}\right)^\alpha \left(\frac{1}{\alpha} - \frac{5}{n^{1/4}}\right) <\sum_{m=k(n)}^{\lfloor n/\ell \rfloor} \frac{1}{C_m} < \left(\frac{n}{\ell}\right)^\alpha \left(\frac{1}{\alpha}+ \frac{5}{n^{1/4}}\right)$.
\end{cor}

We defer the straightforward proofs of \refL{lem:bd} and \refC{cor:bd2} to \refS{sec:moment}.

For all $k\in\N$, let $\cB_k^+$ be the Borel $\sigma$-algebra of $\cT_k$ (i.e., of $(\cT_k(n),\dleb)$, in view of \refP{prop:dist}). Note that conditioned on $\cT_k$, we do {\em not} know any information of the embellished vertices. Given $S\in \cB_k^+$, write $|S|=\mu_k(S)\cdot C_k$; so when $S$ is a path, $|S|$ is the intrinsic path-length. For all integer $n\ge (k-1)\ell$, let $M(S,n)$ be the number of vertices of $\cT_k(n)$ on $S$, and write 
\[
\widehat{M}(S,n) = \frac{c}{n^\alpha}\cdot M(S,n).
\]

\begin{fact}
\label{fact:bin}
Fix $k\in\N$ and $S\in\cB_k^+$. Then for all integer $j\ge k$, given $\cT_k$ and $C_j$,
\[
M(S,j\ell) - M(S,(j-1)\ell) \sim \mathrm{Binomial}\left(\ell,\frac{|S|}{C_j}\right),
\]
and given $C_k,C_{k+1},\ldots$, the variables $\bklg{M(S,j\ell) - M(S,(j-1)\ell)}_{j\geq k}$ are
independent. 
\end{fact}

We first define a nice event, then prove an exponential bound given such an event. For all integers $k,n$ with $n\ge k\ell$, define the event $F_{k,n}$ as
\begin{equation}\label{eq:Fkn}
F_{k,n}:=\left\{\sum_{m=k}^{\lfloor n/\ell \rfloor} \frac{1}{C_m} \in \left( \left(\frac{n}{\ell}\right)^\alpha\cdot \left(\frac{1}{\alpha}\pm \frac{5}{n^{1/4}}\right)\right)\right\}
\bigcap \left\{ \left\vert C_{k}-k^{\frac{1}{\ell+1}}\right\vert < 10 k^{\frac{1}{\ell+1}-\frac14}\right\}.
\end{equation}
Given $k:\N\to\N$ such that $k(n)\to\infty$ and $k(n) = o(n^{1/2})$, by \refC{cor:bd2}, with sufficiently large $n$,
\begin{equation}\label{in:F}
\p{F_{k(n),n}} > 1 - 2 \sum_{m=k(n)}^{\lfloor n/\ell\rfloor} e^{-m^{1/2}/4}.
\end{equation}
Given an event $F$, the notation $F^c$ denotes the complement of $F$.

The next result is the key to the results of this section, which says that the rescaled number of vertices falling into a subset~$S$ of the tree has the same asymptotics as~$\abs{S}$.

\begin{lem}\label{lem:convas}
Let $k:\N\to\N$ be such that $k(n) = \Omega((\log n)^{10})$ and $k(n) = o(n^{1/2})$. Then for sufficiently large $n$, for $\veps = \veps_n > 80 \alpha k(n)^{\frac{1}{\ell+1}} n^{-1/4}$, and for all $S\in \cB_{k(n)}^+$,
\begin{equation}
\label{in:convas0}
\p{\left\vert \widehat{M}(S,n) - |S| \right\vert \ge \veps,~F_{k(n),n}~\bigg\vert~\cT_{k(n)}} \le 2 \exp\left( - \frac{\veps^2 n^\alpha }{32c k(n)^{\frac{1}{\ell+1}}}\right);
\end{equation}
it follows that
\begin{equation}\label{in:convas}
\E{\p{\left\vert \widehat{M}(S,n) - |S| \right\vert \ge \veps~\bigg\vert~\cT_{k(n)}} }\le 2 \exp\left( - \frac{\veps^2 n^\alpha }{32c k(n)^{\frac{1}{\ell+1}}}\right) + e^{-k(n)^{1/3}},
\end{equation}
where the second term $e^{-k(n)^{1/3}}$ comes from $\E{\p{F_{k(n),n}^c\Big\vert \cT_{k(n)}}}$, not depending on $\veps$.
\end{lem}

\begin{proof}
Let $n\in\N$ be sufficiently large (to be made precise) and write $k=k(n)$, $\veps=\veps_n$. Conditioned on $\cT_k$, fix $S\in\cB_{k}^+$. To ease notation, assume that $\ell$ divides $n$. Write $M=M(S,n)$. By Markov's inequality, for any $t>0$,
\begin{equation}\label{in:cher1}
\p{M \ge \frac{n^\alpha}{c} (|S|+\veps),~F_{k,n}~\bigg\vert~ \cT_k}
\le \E{e^{t M}\cdot \I{F_{k,n}}~\bigg\vert ~\cT_k}\cdot  e^{-\frac{t n^\alpha}{c}  (|S|+\veps)}.
\end{equation}
Now we take a closer look at the bound (\ref{in:cher1}). Write $M_0= M(S,(k-1)\ell)$, and for integer $k\le m\le n/\ell$, write $B_m = M(S,m\ell)-M(S,(m-1)\ell)$; so given $C_k,\ldots,C_{n/\ell}$, Fact~\ref{fact:bin} implies the $B_m$'s are independent 
Binomial$(\ell,\frac{|S|}{C_m})$-variables. Then $M = M_0 + \sum_{m=k}^{n/\ell} B_m$. For $t\in(0,1]$, $e^t -1-t\le t^2$ and $\log\left\{1+\frac{|S|}{C_m} (e^t - 1)\right\}\le \frac{|S|}{C_m} (e^t-1)$, so
\begin{align*}
\E{e^{tM} \cdot \I{F_{k,n}}~ \bigg\vert~ \cT_k,M_0,C_k,\ldots,C_{n/\ell}} =&~ \I{F_{k,n}}\cdot e^{tM_0}\cdot \prod_{m=k}^{n/\ell} \E{e^{tB_m}~\bigg\vert ~\cT_k,C_k,\ldots,C_{n/\ell}}\\
=&~ \I{F_{k,n}}\cdot e^{tM_0}\cdot e^{\ell\sum_{m=k}^{n/\ell} \log \left\{1 + \frac{|S|}{C_m} (e^t-1)\right\}}\\
\le&~ \I{F_{k,n}}\cdot e^{tM_0}\cdot e^{\ell(t+t^2)\sum_{m=k}^{n/\ell} \frac{|S|}{C_m}} .
\end{align*}
Note that $M_0\le (\ell+1)k$ and $\I{F_{k,n}}\le1$. Together with (\ref{eq:Fkn}) and (\ref{in:cher1}), by averaging over $C_k,\ldots,C_{n/\ell}$, we have
\begin{align}
&\p{M\ge \frac{n^\alpha}{c} (|S|+\veps),~F_{k,n}~\bigg\vert~\cT_k}
\le \I{F_{k,n}}\cdot e^{(\ell+1)tk + \ell(t+t^2)|S|\left(\frac{n}{\ell}\right)^\alpha \left(\frac1\alpha+\frac{5}{n^{1/4}}\right) - \frac{tn^\alpha}{c} (|S|+\veps)}.\label{in:cher3}
\end{align}
 Recall that $c= \frac{\ell^\alpha}{\ell+1}$ and $\alpha= \frac{\ell}{\ell+1}$, so $\ell t |S| \left(\frac{n}{\ell}\right)^\alpha \frac{1}{\alpha} - \frac{tn^\alpha}{c}|S| = 0$. Hence, by rearrangement and cancellation, the exponent of (\ref{in:cher3}) is simplified as follows:
\begin{align}
&(\ell+1)tk + \ell(t+t^2)|S|\left(\frac{n}{\ell}\right)^\alpha \left(\frac1\alpha+\frac{5}{n^{1/4}}\right) - \frac{tn^\alpha}{c} (|S|+\veps)\label{exponent}\\
=&~ (\ell+1)tk +  (t+t^2) n^{\alpha -1/4}\cdot \frac{5|S|\alpha}{c}+  t^2 n^\alpha \cdot \frac{|S|}{c} -tn^\alpha\cdot  \frac{ \veps}{c}\notag.
\end{align}

Furthermore, since $|S|\le C_k$, $k=k(n)=o(n^{1/2})$, and $\veps = \veps_n > 80 \alpha k(n)^{\frac{1}{\ell+1}} n^{-1/4}$, there exists $n_1\in\N$ such that for all $n\ge n_1$, $(\ell+1)k\cdot c< n^\alpha\cdot \frac{ \veps}{8}$, and, on the event $F_{k,n}$, $n^{\alpha-1/4}\cdot 5|S|\alpha< n^{\alpha -1/4}\cdot 10k^{\frac{1}{\ell+1}}\alpha < n^\alpha\cdot \frac{\veps}{8}$; so $ n^\alpha\cdot \frac{\veps}{2 }- k\cdot c - n^{\alpha - 1/4}\cdot 5\alpha |S|>  n^\alpha\cdot \frac{\veps}{4}$. Below we assume $\cT_k$ is given and $F_{k,n}$ holds. Assume $n\ge n_1$, and take
\begin{equation*}
t = \min\left\{\frac{1}{8},~\frac{ n^\alpha\cdot \veps/2 - k\cdot c - n^{\alpha-1/4} \cdot 5 \alpha|S|}{n^\alpha \cdot |S|+ n^{\alpha -1/4}\cdot 5\alpha|S|}\right\}.
\end{equation*}
It follows that $0<t\le \frac18$. We first consider the case $\frac18>t=\frac{ n^\alpha\cdot \veps/2 - k\cdot c - n^{\alpha-1/4} \cdot 5 \alpha|S|}{n^\alpha \cdot |S|+ n^{\alpha -1/4}\cdot 5\alpha|S|}$: the last equality immediately yields $(\ell+1)tk + (t^2+t)n^{\alpha-1/4}\cdot \frac{5|S|\alpha}{c} + t^2 n^\alpha \cdot \frac{|S|}{c} = t n^\alpha \cdot \frac{\veps}{2c}$, so we easily obtain that
\begin{equation}\label{eq:case1}
(\ell+1)tk +  (t+t^2) n^{\alpha -1/4}\cdot \frac{5|S|\alpha}{c}+  t^2 n^\alpha \cdot \frac{|S|}{c} -tn^\alpha\cdot  \frac{ \veps}{c}
=- t n^\alpha \cdot \frac{\veps}{2c}.
\end{equation}
Moreover, given that $\frac18>t=\frac{ n^\alpha\cdot \veps/2 - k\cdot c - n^{\alpha-1/4} \cdot 5 \alpha|S|}{n^\alpha \cdot |S|+ n^{\alpha -1/4}\cdot 5\alpha|S|} > \frac{n^\alpha\cdot \veps/4}{2n^\alpha\cdot |S|}$, by increasing $n$ if necessary, we have $t> \frac{n^\alpha \cdot \veps/4}{2n^\alpha \cdot |S|} \ge \frac{\veps}{8C_k}>  \frac{\veps}{16 k^{\frac{1}{\ell+1}}}$. So (\ref{eq:case1}) is upper bounded by $-\frac{\veps^2 n^\alpha}{32c k^{\frac{1}{\ell+1}}}$. Next consider the case $t=\frac18\le \frac{ n^\alpha\cdot \veps/2 - k\cdot c - n^{\alpha-1/4} \cdot 5 \alpha|S|}{n^\alpha \cdot |S|+ n^{\alpha -1/4}\cdot 5\alpha|S|} < \frac{\veps}{2|S|}$: in this case $|S|<4\veps$ and, by substituting $t=\frac18$ and using the inequality $(\ell+1)k\cdot c< n^\alpha \cdot \frac{\veps}{8}$ again,
\begin{align*}
&(\ell+1)tk +  (t+t^2) n^{\alpha -1/4}\cdot \frac{5|S|\alpha}{c}+  t^2 n^\alpha \cdot \frac{|S|}{c} -tn^\alpha\cdot  \frac{ \veps}{c}\\
\le&~ n^\alpha \cdot \frac{\veps}{64c} + n^{\alpha-1/4}\cdot  \frac{9}{64}\cdot \frac{20\veps\alpha}{c} + n^\alpha\cdot \frac{4\veps}{64c} - n^\alpha\cdot \frac{\veps}{8c}\\
=&~ -n^\alpha\cdot \frac{3\veps}{64c} + n^{\alpha-1/4}\cdot \frac{45\veps \alpha}{16c} < -\frac{\veps^2 n^\alpha}{32c k^{\frac{1}{\ell+1}}}.
\end{align*}
Together with (\ref{in:cher3}), for sufficiently large $n$ we have
\begin{align*}
&\p{M\ge \frac{n^\alpha}{c}(|S|+\veps),~F_{k,n}~\bigg\vert~\cT_k}\le \exp\left( - \frac{\veps^2 n^\alpha }{32c k^{\frac{1}{\ell+1}}}\right).
\end{align*}
Similarly, we deduce that $\p{ \frac{n^\alpha}{c}|S| - M\ge \frac{n^\alpha}{c}\veps,~F_{k,n}~\bigg\vert~\cT_k} \le \exp\left( - \frac{\veps^2 n^\alpha }{32c k^{\frac{1}{\ell+1}}}\right)$. (\ref{in:convas0}) then follows by the triangle inequality. 

Finally, by averaging over all $\cT_k$, it follows from (\ref{in:F}) that, by increasing $n$ if necessary,
\[
\E{\p{F_{k,n}^c\Big\vert \cT_k}} = \p{F_{k,n}^c} \le 2\sum_{m=k}^{n/\ell} e^{-m^{1/2}/4} \le e^{-k^{1/3}},
\]
where the last inequality is because $k=k(n) = \Omega((\log n)^{10})$. Then it follows by (\ref{in:convas0}) that
\begin{align*}
\E{\p{\left\vert \widehat{M}(S,n) - |S| \right\vert \ge \veps~\bigg\vert~\cT_{k}} } \le&~ \E{\p{\left\vert \widehat{M}(S,n) - |S| \right\vert \ge \veps,~F_{k,n}~\bigg\vert~\cT_{k} }}
+ \E{\p{F_{k,n}^c\Big\vert \cT_k}}\\
\le&~2\exp\left( - \frac{\veps^2 n^\alpha }{32c k^{\frac{1}{\ell+1}}}\right) + e^{-k^{1/3}}.
\end{align*}
Notice that the second term $e^{-k^{1/3}}$ in the bound does not depend on $\veps$.
\end{proof}

\refL{lem:convas} easily leads to the Gromov-Hausdorff (GH) version of \refP{prop:finite}; see (\ref{in:distortion}) for argument. To extend the result to GHP convergence (see \refS{sec:gh}),
we need to consider the measures on the trees. First we simplify notation; given appropriate $k:\N\to\N$, write 
\[
\cT^n = \cT_{k(n)}(n).
\]
We need to bound the minimal discrepancy with respect to the uniform probability measures $\nu_{k(n),n}$ on $(v(\cT^n),\rd_n)$ and $\mu_{k(n)}$ on $(\cT^n,\dleb)$. To accomplish that, we show that $\nu_{k(n),n}$ and $\mu_{k(n)}$ are close. 
 
\begin{cor}\label{cor:embellish0}
Fix $\veps>0$. Let $k:\N\to\N$ be such that $k(n)=\Omega\left((\log n)^{10}\right)$ and $k(n)=o(n^{1/10})$. Then for all $S\in \cB_{k(n)}^+$, as $n\to\infty$,
\[
\E{\p{\left\vert \nu_{k(n),n}(v(S)) - \mu_{k(n)}(S)\right\vert>\frac{\veps}{11k(n)^{\frac{2}{\ell+1}}}~\bigg\vert~\cT_{k(n)}}} = o(n^{-3}),
\]
where the rate of decay does not depend on $S$.
\end{cor}

\begin{proof}
Fix sufficiently large $n\in\N$ and write $k=k(n)$. Let $S\in \cB_{k}^+$. Recall that $|S| = \mu_{k}(S)\cdot C_{k}$, $|\nu_{k,n}(v(S))|\le 1$, and 
\[
\widehat{M}(S,n) =\frac{c}{n^\alpha} \cdot M(S,n) = \frac{c}{n^\alpha} \cdot \nu_{k,n}(v(S))\cdot |v(\cT^n)|.
\]
Next, note that, conditioned on $\cT_k$, the event $\left\{\left\vert \nu_{k,n}(v(S)) - \mu_{k}(S) \right\vert >2\veps\right\}$ is a subset of the union of the events
\begin{align*}
&\left\{\left\vert \nu_{k,n}(v(S))\cdot \frac{c\cdot |v(\cT^n)|}{n^\alpha}  - \mu_{k}(S) \cdot C_{k}\right\vert >\veps\cdot C_{k}\right\} \bigcup
\left\{\left\vert \frac{c\cdot |v(\cT^n)|}{n^\alpha}  -  C_{k}\right\vert \cdot \nu_{k,n}(v(S)) >\veps\cdot C_{k} \right\}.
\end{align*}
On the event $F_{k,n}$ from (\ref{eq:Fkn}), $C_k>1$. It then follows from the triangle inequality that
 \begin{align*}
& \E{\p{\left\vert \nu_{k,n}(v(S)) - \mu_{k}(S) \right\vert >2\veps~\bigg\vert~\cT_{k}}}\\
&\qquad \le ~ \E{\p{\left\vert \widehat{M}(S,n) - |S| \right\vert > \veps,~F_{k,n}~\bigg\vert~\cT_{k}}} \\
&\qquad\instep + \E{ \p{\left\vert \frac{c\cdot |v(\cT^n)|}{n^\alpha}  - C_{k}\right\vert >\veps,~F_{k,n}~\bigg\vert~\cT_{k}}} + 2\E{\p{F_{k,n}^c\Big\vert\cT_k}}\\
&\qquad \le~4 \exp\left( - \frac{\veps^2 n^\alpha }{32c k^{\frac{1}{\ell+1}}}\right)+ 2e^{-k^{1/3}} ;
 \end{align*}
the last inequality follows by applying \refL{lem:convas} twice. Now, replacing $\veps$ by $\frac{\veps}{2\cdot 11k^{\frac{2}{\ell+1}}}$ in the above inequality and noticing that given the event $F_{k,n}$, $\frac{\veps}{11k^{\frac{2}{\ell+1}}}$ satisfies the assumption in \refL{lem:convas}, we obtain
\[
\E{\p{\left\vert \nu_{k,n}(v(S)) - \mu_{k}(S) \right\vert >\frac{\veps}{11k^{\frac{2}{\ell+1}}}~\bigg\vert~\cT_{k}}}
\le 4 \exp\left(\frac{-\veps^2n^\alpha}{32\cdot 2^2\cdot 11^2 \cdot c k^{\frac{5}{\ell+1}}} \right)+2 e^{-k^{1/3}}.
\]
 The lemma then follows from that $k=k(n)=\Omega((\log n)^{10})$ and $k(n) =o(n^{1/10})$.
\end{proof}

It may be helpful to recall the definitions relating to GHP convergence in \refS{sec:gh}  before reading the next proof.

\begin{proof}[{\bf Proof of \refP{prop:finite}}]
For most part of the proof we fix a large enough $n$ and write $k=k(n)$ for simplicity, unless we consider varying $n$. Let $\veps_n = k^{-\frac{1}{\ell+1}}$. Since the total length of $\cT^n$ is $C_{k}$, we may cover $\cT^n$ by $M_n\coloneqq \left\lceil \veps_n^{-1}C_k\right\rceil$ balls, denoted by $B_{n,1},\ldots,B_{n,M_n}$, with diameter at most $\veps_n$. Let $A_{n,1} = B_{n,1}$, and for $i>1$, let $A_{n,i} = B_{n,i} \setminus A_{n,i-1}$. Then $\{A_{n,1},\ldots,A_{n,M_n}\}$ is a covering of $(\cT^n,\dleb)$ by disjoint sets with diameter at most $\eps_n$. 

Next, define  $S_n = \bigcup_{i=1}^{M_n} v(A_{n,i})\times A_{n,i}$, then $S_n$ is a correspondence between $v(\cT^n)$ and $\cT^n$. Moreover, for each $1\le i\le M_n$, let $w_i$ be the element of $v(A_{n,i})$ such that $w_i$ is closest to the root of $\cT^n$. The distortion of $S_n$ can be bounded as follows:
\begin{align*}
\dis&_n \coloneqq~ \dis(S_n;\rd_n,\dleb) \\
=&~ \sup\left\{|\rd_n(x,y) - \dleb(x',y')|: (x,x')\in S_n, (y,y')\in S_n\right\}\\
=&~ \max_{1\le i\le j\le M_n}
\sup\left\{|\rd_n(x,y) - \dleb(x',y')|: (x,x') \in v(A_{n,i})\times A_{n,i} , (y,y') \in v(A_{n,j}) \times A_{n,j} \right\}\\
\le&~ \max_{1\le i\le j\le M_n}
\sup\big\{ | \rd_n(w_i,w_j) - \dleb(w_i,w_j)| + \rd_n(w_i,x) + \rd_n(w_j,y)\\
&\instep + ~\dleb(w_i,x') + \dleb(w_j,y'): (x,x') \in v(A_{n,i})\times A_{n,i}, (y,y') \in v(A_{n,j}) \times A_{n,j} \big\}\\
\le&~ \max_{1\le i\le j\le M_n} \left\vert\rd_n(w_i,w_j) - \dleb(w_i,w_j)\right\vert +2\veps_n
+\frac{2c}{n^\alpha}\sup_{1\leq i \leq M_n} v(A_{n,i}).
\end{align*}
Now, given $x,y\in v(\cT^n)$, write $[x,y)$ for the path in $\cT^n$ from $x$ (included) to $y$ (excluded). So $\rd_n(x,y) = \widehat{M}([x,y),n)$ and $\dleb(x,y) = |[x,y)| $. So
\[
\dis_n\le \max_{1\le i\le j\le M_n} \left\vert \widehat{M}([w_i,w_j),n) - |[w_i,w_j)| \right\vert+ 2\max_{1\leq i \leq M_n}\left\vert \widehat{M}(A_{n,i},n) - |A_{n,i}| \right\vert +2\veps_n.
\]
Recall the definition of the event $F_{k,n}$ from (\ref{eq:Fkn}) and note that on this event, we have $M_n\le m_n\coloneqq \left\lceil\veps_n^{-1}k^{\frac{1}{\ell+1}} \left(1+10 k^{-\frac14}\right)\right\rceil = \left\lceil k^{\frac{2}{\ell+1}}  \left(1+10 k^{-\frac14}\right)\right\rceil $. 
Then for any $\veps>0$, it follows that
\begin{align*}
\p{\dis_n>4\veps_n+3\veps}
\le&~ \p{\dis_n>2\veps_n+3\veps,~F_{k,n}} + \p{F_{k,n}^c}\\
\le&~ \p{\max_{1\le i\le j\le M_n} \left\vert \widehat{M}([w_i,w_j),n) - |[w_i,w_j)| \right\vert > \veps,~F_{k,n}} \\
&\quad+\p{\max_{1\le i\le M_n} \left\vert \widehat{M}(A_{n,i},n) - |A_{n,i}| \right\vert > \veps,~F_{k,n}}+\p{F_{k,n}^c}\\
\le &~ \E{\sum_{1\le i\le j\le M_n} \p{\left\vert \widehat{M}([w_i,w_j),n) - |[w_i,w_j)| \right\vert > \veps,~F_{k,n}}} \\
&\quad+\E{ \sum_{1\leq i \leq M_n} \p{\left\vert \widehat{M}(A_{n,i},n) - |A_{n,i}| \right\vert > \veps,~F_{k,n}}}+\p{F_{k,n}^c}.
\end{align*}
Applying \refL{lem:convas} with $S = [w_i,w_j)$ and $S=A_{n,i}$ yields that
\[
\p{\dis_n>2\veps_n+3\veps}
\le 4 m_n^2\exp\left( - \frac{\veps^2 n^\alpha}{32c k(n)^{\frac{1}{\ell+1}}}\right) + e^{-k(n)^{1/3}}.
\]
Noting that $\veps_n\to0$, $k(n) = \Omega\left((\log n)^{10}\right)$, $k(n) = o\left(n^{1/10}\right)$, and $m_n< 11 k(n)^{\frac{2}{\ell+1}}$, it is easily seen that $\sum_{n\in\N} \p{\dis_n>2\veps_n+3\veps}<\infty$, and so by the Borel-Cantelli lemma,
\begin{equation}\label{in:distortion}
\dis_n \to0 \mbox{ a.s.},
\end{equation}
and the GH convergence is shown.

To show GHP convergence, we follow \cite[Proof of Proposition 4.8]{ABGM} and define $\pi^\circ_n$ on the product space $v(\cT^n)\times \cT^n$ as follows. Given $1\le i\le M_n$, for Borel sets $X\subset v(A_{n,i})$ of $(v(\cT^n),\rd_n)$ and $Y\subset A_{n,i}$ of $(\cT^n,\dleb)$, define
\[
\pi^\circ_n(X,Y) = \frac{\nu_{k,n}(X) \cdot \mu_{k}(Y)}{\max\left\{ \nu_{k,n}(v(A_{n,i})), \mu_{k}(A_{n,i})\right\}}.
\]
For $i\neq j$, let $\pi^\circ_n(v(A_{n,i}),A_{n,j}) = 0$; so 
\begin{equation}\label{eq:zeromeas}
\pi^\circ_n(S_n^c)=0.
\end{equation}
Such rectangles $X\times Y$ form a $\pi$-system generating the product $\sigma$-algebra, so $\pi^\circ_n$ extends uniquely to a measure $\pi_n$ on the product $\sigma$-algebra of $(v(\cT^n),\rd_n)$ and $(\cT^n,\dleb)$.

Now we derive the discrepancy $\rD_n\coloneqq \rD(\pi_n;\nu_{k,n},\mu_{k})$ of $\pi_n$ with respect to $\nu_{k,n}$ and $\mu_{k}$. Note that $\pi_n(v(A_{n,i}),A_{n,i}) = \min\left\{ \nu_{k,n}(v(A_{n,i})), \mu_{k}(A_{n,i})\right\}$. Writing $p$ and $p'$ for the projections of $v(\cT^n)\times\cT^n$ to the first and the second coordinates respectively, an easy calculation shows that
\begin{align*}
\rD_n =&~ \|\nu_{k,n} - p_* \pi_n\| + \|\mu_{k} - p'_* \pi_n\|\\
=&~ \sum_{i=1}^{M_n} \left[ \nu_{k,n}(v(A_{n,i})) - \min\left\{ \nu_{k,n}(v(A_{n,i})), \mu_{k}(A_{n,i})\right\}\right]\\
&\instep + \sum_{i=1}^{M_n} \left[ \mu_{k}(A_{n,i}) - \min\left\{ \nu_{k,n}(v(A_{n,i})), \mu_{k}(A_{n,i})\right\} \right]\\
=&~ \sum_{i=1}^{M_n} \left\vert \nu_{k,n}(v(A_{n,i})) - \mu_{k}(A_{n,i}) \right\vert.
\end{align*}
Note that on the event~$F_{k,n}$,  $M_n\leq m_n$, so for any $\veps>0$, 
\begin{align*}
\p{\rD_n>\veps} 
\le &~ \E{\p{\sum_{i=1}^{M_n} \left\vert \nu_{k,n}(v(A_{n,i})) - \mu_{k}(A_{n,i}) \right\vert>\veps, F_{k,n}~\bigg\vert ~\cT_{k}}} + \p{F_{k,n}^c} \\
\le&~  \E{\I{M_n\leq m_n}\sum_{i=1}^{M_n} \p{\left\vert \nu_{k,n}(v(A_{n,i})) - \mu_{k}(A_{n,i}) \right\vert>\frac{\veps}{m_n}, F_{k,n}~\Big\vert ~\cT_{k}}} + \p{F_{k,n}^c}.
\end{align*}
We now use the notation $k(n)$ to emphasize that $k(n)$ changes with $n$ and note that $m_n < 11k(n)^{\frac{2}{\ell+1}}$ and $k(n)=o(n^{1/10})$. Summing over $n\in\N$ and applying \refC{cor:embellish0} then yields that
\[
\E{\I{M_n\leq m_n}\sum_{i=1}^{M_n} \p{\left\vert \nu_{k,n}(v(A_{n,i})) - \mu_{k}(A_{n,i}) \right\vert>\veps/M_n, F_{k,n}~\Big\vert ~\cT_{k}}}\leq m_n \cdot o(n^{-3}),
\]
and so this combined with Lemma~\ref{lem:convas} to bound $\IP(F_{k,n}^c)$ yields
\[
\sum_{n\in\N}\p{\rD_n>\eps} \leq \sum_{n\in\N}\left( m_n \cdot o(n^{-3})+e^{-k(n)^{1/3}}\right)<\infty.
\]
 Hence, 
\begin{equation}\label{conv:desc}
\rD_n\to0 \mbox{ a.s.}
\end{equation}
Finally, note that for all $n\in\N$, 
\[
\dghp\left(\left(v(\cT^n),\rd_n,\nu_{k(n),n}\right),
\left(\cT^n,\dleb,\mu_{k(n)}\right)\right)\le \max\left\{ \dis_n/2, \rD_n, \pi_n(S_n^c)\right\}.
\]
 It follows from (\ref{in:distortion}), (\ref{eq:zeromeas}), and (\ref{conv:desc}) that, a.s.,
\begin{align*}
\dghp\left(\left(v(\cT^n),\rd_n,\nu_{k(n),n}\right),
\left(\cT^n,\dleb,\mu_{k(n)}\right)\right)
\to 0.
\end{align*}
Since we are working in the probability space where the equalities of \refP{prop:dist} hold, the proof is completed.
\end{proof}

\section{Tightness property}\label{sec:tight}

In \refS{sec:polya}, we describe how the combinatorial tree $\rT(n)$ relates to an infinite-colors P\'olya urn, which helps us analyse the heights and sizes of subtrees in $\rT(n)$. In \refS{sec:pftight}, we establish \refP{prop:tight}, with the proofs of several lemmas deferred to the subsequent subsections.

\subsection{An infinite-colors P\'olya urn}\label{sec:polya}

At time $0$, an urn contains only one ball of color $1$. At time $n\in\N$, pick a ball from the urn uniformly at random, return the ball to the urn along with another ball of the same color. In addition, if $\ell$ divides $n$, and if the urn contains balls of colors $1,\ldots,k-1$, then an additional ball of color $k$ is added to the urn. For $n,k\in\N$ with $n\ge (k-1)\ell$, let $U_k(n)$ be the number of balls of color $k$ at time $n$, and let $M_k(n) = U_1(n)+\ldots + U_k(n)$. Note that at time $k\ell$, there are $(\ell+1)k+1$ balls of colors $1,\ldots,k+1$ (the extra $1$ accounts for the initial ball of color $1$), and there is only $1$ ball of color $k+1$.

Recall the construction of the combinatorial tree $\rT(n)$ from \refS{sec:intro}. For all $k\in\N$, $v_{k\ell}$ is a branchpoint, i.e., a vertex with degree at least $3$. For all $k, n\in\N$ with $n\ge (k-1)\ell$, we call the (graph-theoretic) path in $\rT(n)$ from $v_{k\ell}$ to the leaf $L_{1+k}$ {\em branch $k$}. The {\em length} of a path in $\rT(n)$ is the number of (graph-theoretic) edges in it. Note that for $k,n\in\N$ with $n\ge (k-1)\ell$, the length of branches $1,\ldots,k$ in $\rT(n)$ have the same law as $(U_1(n),\ldots,U_k(n))$, and $(M_1(n),\ldots,M_k(n))$ have the same law as the number of edges in $(\rT_1(n),\ldots,\rT_k(n))$. We may and shall use $U_k(n)$ to denote the length of branch $k$ in $\rT(n)$.

\subsection{Outline and proof for \refP{prop:tight}}\label{sec:pftight}

We first outline the essential step to prove the GH version of \refP{prop:tight}: to obtain a height-bound for the subtrees of $\rT(n)$ pendant to $\rT_k(n)$, where $\rT_k(n)$ is the subtree of $\rT(n)$ spanned by the root and the first $k$ leaves. To accomplish this, we express the height-bound of the subtrees in terms of $\sum_{i=k(n)+1}^{\lfloor n/\ell\rfloor+1} \frac{(C_i - C_{i-1})U_i(n)}{C_i}$, in \refL{lem:expbd}; then deduce a bound for this sum, in \refL{lem:sum}. Write $\cF_{k,n}$ for the $\sigma$-algebra generated by $C_{k},\ldots,C_{\lfloor n/\ell\rfloor+1},U_{k+1}(n),\ldots,U_{\lfloor n/\ell\rfloor+1}$. For all $i\in\N$ write $\Delta C_i = C_i - C_{i-1}$.

\begin{lem}\label{lem:expbd}
Fix $n,k\in\N$ with $n\ge  k\ell$, and let $u$ be a uniformly chosen vertex from   $v\left(\rT(n)\right)\setminus v\left(\rT_{k}(n)\right)$. Then for positive $\lambda\le \left\{ \max_{k+1\le i\le \lfloor n/\ell\rfloor+1} U_i(n)\right\}^{-1}$,
\[
\E{\exp\left(\lambda \cdot \dgr(u, \rT_k(n))\right)~\Big\vert~ \cF_{k,n}} \le \exp\left(\lambda\ell+5\lambda \sum_{i=k+1}^{\lfloor \frac{n}{\ell}\rfloor+1} \frac{\Delta C_i U_i(n)}{C_i}\right).
\]
\end{lem}

\begin{lem}\label{lem:sum}
Let $\veps\in(0,1)$. Let $k:\N\to\N$ be such that $k(n)=\Omega\left(n^{1/100}\right)$ and $k(n)=o\left(n^{\frac{\ell}{2\ell+1}}\right)$. Then
\begin{equation*}
\sum_{n\in\N}n\cdot \p{\sum_{i=k(n)+1}^{\floor{\frac{n}{\ell}}+1} k(n)^{\alpha-\veps} \frac{\Delta C_i U_i(n)}{n^{\alpha}C_i}>1 \mbox{ or} \max_{k(n)+1\le i\le \lfloor \frac{n-1}{\ell}\rfloor+1} U_i(n) > n^\alpha k(n)^{-\alpha+\veps}}
<\infty.
\end{equation*}
\end{lem}

We defer the proofs of Lemmas \ref{lem:expbd} and \ref{lem:sum} to Sections \ref{sec:pf} and \ref{sec:branch}, respectively. They lead to a tightness property of $(\rT_k(n):k\in\N, n\ge k\ell)$, i.e., the GH version of \refP{prop:tight}. To wit, for all $k,n\in\N$ with $n\ge k\ell$, let
\[
D_{k,n} = \frac{c}{n^\alpha}\cdot \max\left\{\dgr\left(w,\rT_{k}(n)\right):w\in v(\rT(n))\setminus v(\rT_{k}(n))\right\}.
\]
It follows from the definition of GH distance that 
\[
\dgh\left(\big(v(\rT_{k}(n)),\frac{c}{n^\alpha}\cdot \dgr\big),\big(v(\rT(n)),\frac{c}{n^\alpha}\cdot \dgr\big)\right)\le D_{k,n}.
\]
We use Lemmas~\ref{lem:expbd} and~\ref{lem:sum} below to show that for an appropriate sequence of increasing $k(n)$,~$D_{k(n),n}\to0$ a.s.; details are given in the proof of \refP{prop:tight}.

Next, we take the measures into consideration, and examine GHP convergence of \refP{prop:tight}. We start by stating a fact about the GHP distance between subspaces that follows in a straightforward way from constructions and definitions; more general statements appear in \citep[Fact 6.4]{ABW} and \citep[Fact 8.6]{Wen}. For all $k,n\in\N$ with $n\ge (k-1)\ell$, let $\overline{\nu}_{k,n}$ be the projection of the uniform probability measure $\nu_n$ of $v(\rT(n))$ onto $v(\rT_k(n))$, i.e., for any $w\in v(\rT_k(n))$, write $\overline{w}$ for the maximal subset of $v(\rT(n))$ such that the removal of $w$ disconnects $\overline{w}$ from $\rT_k(n)$, and let $\overline{\nu}_{k,n}(w) = \nu_n(\overline{w}\cup\{w\})$. Write $\widehat{\bT}(n)=\left(v(\rT(n)),\frac{c}{n^\alpha}\cdot \dgr,\nu_n\right)$, $\overline{\bT}_{k}(n) = \left(v(\rT_{k}(n)),\frac{c}{n^\alpha}\cdot \dgr,\overline{\nu}_{k,n}\right)$, and $\widehat{\bT}_{k}(n) = \left(v(\rT_{k}(n)),\frac{c}{n^\alpha}\cdot \dgr,\nu_{k,n}\right)$. 

\begin{fact}\label{fact:ghp}
For all $k,n\in\N$ with $n\ge(k-1)\ell$, 
\[
\dghp\left(\widehat{\bT}(n), \overline{\bT}_k(n)\right) \le D_{k,n}.
\]
\end{fact}

Upon showing that $D_{k(n),n}\to0$ a.s. for an appropriate $k(n)$, to prove \refP{prop:tight} it suffices to bound $\dghp\left(\overline{\bT}_{k}(n),\widehat{\bT}_{k}(n)\right)$. Note that $\overline{\bT}_{k}(n)$ and $\widehat{\bT}_{k}(n)$ differ only in their measures, so $\dghp\left(\overline{\bT}_{k}(n),\widehat{\bT}_{k}(n)\right)=d_{k,n}(\overline{\nu}_{k,n},\nu_{k,n})$, where $d_{k,n}$ denotes the L\'evy-Prokhorov distance on the metric space $(v(\rT_{k}(n)),\frac{c}{n^\alpha}\cdot \dgr)$.
We show the following lemma in \refS{sec:prok}. 

\begin{lem}\label{lem:prok}
Let $k:\N\to\N$ be such that $k(n)=\Omega\left(n^{1/100}\right)$ and $k(n)=o\left(n^{1/3}\right)$ and assume $\ell\geq 2$. Then, almost surely as $n\to\infty$,
\[
d_{k(n),n}\left(\overline{\nu}_{k(n),n},\nu_{k(n),n}\right) \to0.
\]
\end{lem}

We can now make the discussion above into a precise proof.
\begin{proof}[{\bf Proof of \refP{prop:tight}}]
Fix a sufficiently large $n$ and write $k=k(n)$, until near the end of the proof when we let $n$ vary. It then follows from \refFt{fact:ghp} that
\begin{align}
\dghp\left(\widehat{\bT}(n),\widehat{\bT}_{k}(n)\right)
\le&~ \dghp\left(\widehat{\bT}(n),\overline{\bT}_{k}(n)\right)
+ \dghp\left(\overline{\bT}_{k}(n),\widehat{\bT}_{k}(n)\right)\notag\\
\le&~ D_{k,n} + d_{k,n}(\overline{\nu}_{k,n},\nu_{k,n}).\label{in:diamprok}
\end{align}
Now, we deduce a bound for $D_{k,n}$. Set $\veps=\alpha/4$, and define the event
\[
E_n = \left\{ \sum_{i=k+1}^{\lfloor \frac{n}{\ell}\rfloor+1}k^{\alpha-\veps} \frac{\Delta C_i U_i(n)}{n^\alpha C_i} \le 1\right\}\bigcap \left\{ \max_{k+1\le i\le \lfloor \frac{n}{\ell}\rfloor+1} U_i(n) \le n^\alpha k^{-\alpha+\veps}\right\}.
\]
Choose $u$ uniformly at random from $v\left(\rT(n)\right)\setminus v\left(\rT_{k}(n)\right)$, and take $\lambda(n) =n^{-\alpha}k^{\alpha-\veps}>0$ in \refL{lem:expbd}, noticing that, on the event $E_n$, $\lambda(n)\le \left\{ \max_{k+1\le i\le \lfloor \frac{n}{\ell} \rfloor+1}U_i(n)\right\}^{-1}$. It then follows from Markov inequality and \refL{lem:expbd} that
\begin{align*}
\p{\dgr(u,\rT_{k}(n))\ge n^\alpha k^{-\alpha+2\veps}}
\le&~\p{\dgr(u,\rT_{k}(n))\ge n^\alpha k^{-\alpha+2\veps},~E_n} +\p{E_n^c}\\
\le&~ \frac{\E{\I{E_n}\cdot \E{\exp\left(\lambda(n)\cdot \dgr(u,\rT_{k}(n))\right)~\Big\vert~\cF_{k,n}}}}{\exp\left(\lambda(n) n^\alpha k^{-\alpha+2\veps}\right)} + \p{E_n^c}\\
\le&~ \frac{\exp\left(\lambda(n)\ell + 5\lambda(n)\cdot n^\alpha k^{-\alpha+\veps}\right)}{\exp\left(\lambda(n) \cdot n^\alpha k^{-\alpha+2\veps}\right)} +\p{E_n^c}\\
\le &~\exp\left(\ell+ 5- k^{\veps}\right) + \p{E_n^c}.
\end{align*}
Recall that $\veps=\alpha/4$, so $-\alpha/2=-\alpha+2\veps$. It then follows from a union bound that
\begin{align*}
\p{D_{k,n}\ge c\cdot k^{-\alpha/2}} 
\le &~\E{\sum_{w\in v(\rT(n))\setminus v(\rT_{k}(n))} \p{\dgr(w,\rT_k(n))\ge n^\alpha k^{-\alpha+2\veps}}}\\
\le&~ 2n \cdot \left\{\exp\left(\ell+5- k^{\veps}\right) + \p{E_n^c}\right\}.
\end{align*}
Now we use the notation $k(n)$ and sum over $n$ on both sides of the above inequality:
\[
\sum_{n\in\N}\p{D_{k(n),n}\ge c\cdot k(n)^{-\alpha/2}} 
\le \sum_{n\in\N}2n\cdot \left\{\exp\left(\ell+ 5- k(n)^{\veps}\right) + \p{E_n^c} \right\}.
\]
Since $k(n)=\Omega\left(n^{1/100}\right)$, $\sum_{n\in\N}n\cdot \exp\left(\ell+5- k(n)^{\veps}\right) <\infty$. Moreover, \refL{lem:sum} implies that $\sum_{n\in\N}n\cdot \p{E_n^c}<\infty$. It then follows from the Borel-Cantelli lemma that $D_{k(n),n}\to0$ a.s.. Together with (\ref{in:diamprok}) and \refL{lem:prok}, we may conclude the proof.
\end{proof}
 
We only have left to prove Lemmas~\ref{lem:expbd},~\ref{lem:sum}, and~\ref{lem:prok}, which we do in the forthcoming Sections~\ref{sec:pf},~\ref{sec:branch}, and~\ref{sec:prok}, respectively.

\subsection{Height bound}\label{sec:pf}

In this subsection we prove \refL{lem:expbd}. Recall from \refS{sec:intro} that, $\cT'_k$ is obtained from inserting $\ell$ vertices uniformly over $\cT((k-1)\ell)$, i.e., $\cT'_k$ is $\cT(k\ell)$ without the $(k+1)$:th branch. Now, for all $k\in\N$, let $X_{k}$ be the last inserted vertex of $\cT'_k$; so $X_k$ has the uniform law over $\cT'_k$ with respect to the Lebesgue measure.

Next, we construct a sequence of new {\em embellished} trees $(\cT^\circ_k:k\in\N)$, coupled with a sequence of vertices $(X^\circ_k:k\in\N)$, such that $(\cT^\circ_k,X^\circ_k)\eqdist (\cT'_k,X_k)$. Our construction is a variant of the one in \citep[Section 1.2]{CH}.

\vspace*{.3cm}
\noindent{{\bf A coupling.}}
Let $(W_k,V_k:k\in\N)$ be i.i.d. Uniform$(0,1)$-variables. We construct $\cT^\circ_1$ by (1) inserting $\ell-1$ vertices at uniform points over a branch of length $C_1$; and (2) let $X_1^\circ$ be the point at distance $V_1 C_1$ from a fixed endpoint $X^\circ_0$ of the branch. Given the pairs $(\cT^\circ_i,X^\circ_i)$ for $i=1,\ldots,k$ for some $k\in\N$, we construct $(\cT^\circ_{k+1},X^\circ_{k+1})$ as follows. Note that, before $(\cT^\circ_{k+1},X^\circ_{k+1})$ is constructed, we do not know yet whether to view $X_k^\circ$ as a vertex (in the upcoming case (a)) or just a point (case (b)). The reason to emphasize the difference between vertices and points is to align with the distribution of $\cT'_k$, which is viewed as a union of a real tree and vertices, and the last $\ell$ vertices in $\cT_k'$ each has $1/\ell$ probability of becoming a junction vertex in $\cT(k\ell)$, but a random point has $0$ probability of becoming a juntion. Recall that $\Delta C_{k+1} = C_{k+1} - C_k$.
\begin{itemize}
\item[(a)] If $W_{k+1}\le \frac{\Delta C_{k+1}}{C_{k+1}}$, then let $\cT^\circ_{k+1}$ be obtained from $\cT^\circ_k$ by (1) attaching a branch of length $\Delta C_{k+1}$ to $X^\circ_k$; (2) inserting a vertex, denoted by $X^\circ_{k+1}$, in the latest branch at distance $\Delta C_{k+1} V_{k+1}$ from $X^\circ_k$; and (3) inserting $\ell-1$ vertices at random points of the existing tree, uniform for the Lebesgue measure. We view $X^\circ_k$ as a vertex in this case.
\item[(b)] If $W_{k+1}> \frac{\Delta C_{k+1}}{C_{k+1}}$, then let $\cT^\circ_{k+1}$ be obtained from $\cT^\circ_k$ by (1) inserting a vertex at a random point of $\cT^\circ_k$, uniform for the Lebesgue measure; (2) attaching a branch of length $\Delta C_{k+1}$ to this last inserted vertex of $\cT^\circ_k$; and (3) inserting $\ell-1$ vertices at random points of the existing tree, uniform for the Lebesgue measure. Let $X_{k+1}^\circ = X_k^\circ$, viewed as a random point rather than a vertex.
\end{itemize}
Note that the projection of $X^\circ_j$ to $\cT^\circ_k$ is $X^\circ_k$ for all integers $j\ge k\ge1$.

\begin{lem}\label{lem:unif}
For all $k\in\N$, $X_k^\circ$ and $X_k$ are respectively uniform over $\cT^\circ_k$ and $\cT'_k$ for the Lebesgue measure, and $(\cT^\circ_k,X^\circ_k)\eqdist (\cT'_k,X_k)$.
\end{lem}

\begin{proof}
First note that $X_k$, the last inserted vertex of $\cT'_k$, is uniform over $\cT'_k$ for the Lebesgue measure. Next, we show by induction on $k$ that $X^\circ_k$ has the uniform law over $\cT^\circ_k$ for the Lebesgue measure. Base case $k=1$ is trivially verified. Given that $X^\circ_k$ is uniform over $\cT^\circ_k$ for some $k\in\N$, since $X^\circ_{k+1}$ has $\frac{\Delta C_{k+1}}{C_{k+1}}$ probability of landing on a uniform location of branch $k+1$, with the complement probability of being $X^\circ_k$ which is uniform on $\cT^\circ_k$, it is clear that $X^\circ_{k+1}$ has the uniform law over $\cT^\circ_{k+1}$ for the Lebesuge measure. 

Furthermore, it is easily seen that the $(k+1)$:th branch is attached to a uniform point of $\cT^\circ_k$, for the Lebesgue measure, for all $k\in\N$ (step (1) of case (a): $X_k^\circ$ is uniform over $\cT^\circ_k$; step (2) of case (b): the last inserted vertex of $\cT^\circ_k$ is uniform over $\cT^\circ_k$). Moreover, given $\cT^\circ_k$, the first vertex to be inserted has $\frac{\Delta C_{k+1}}{C_{k+1}}$ probability of landing on a uniform location of the $(k+1)$:th branch (step (2) of case (a)), with the complement probability of landing on a uniform point of $\cT^\circ_k$ (step (1) of case (b)). Also, the next $\ell-1$ vertices to be inserted are uniform over the existing tree (step (3) in both cases). We thus have $\cT^\circ_{k+1}\eqdist \cT'_{k+1}$. It follows by induction that $\cT^\circ_k\eqdist \cT'_k$ for all $k\in\N$. Since both $X^\circ_k$ and $X_k$ are respectively uniform over these two trees for the Lebesgue measure, we have $(\cT^\circ_k,X^\circ_k)\eqdist (\cT'_k,X_k)$.
\end{proof}

For ease of notation, fix $k,n\in\N$ with $n\ge k\ell$, and let $m$ be the largest integer such that $n\ge m\ell$. For all integer $1\le i\le m$, given that $W_i\le \frac{\Delta C_i}{C_i}$, write $S_i$ for the path $[X^\circ_{i-1},X^\circ_i)$ in $\cT^\circ_m$, and write $M^\circ(S_i,m)$ for the number of vertices on $S_i$ in $\cT^\circ_m$. Let $\cT^\circ_{k,m}$ (resp. $\cT'_{k,m}$) be the subtree of $\cT^\circ_m$ (resp. $\cT'_m$) spanned by the root and the first $k$ leaves. Denote by $E^\circ$ the event that $X^\circ_m\notin \cT^\circ_{k,m}$, and analogously denote by $E$ the event that $X_m\notin \cT_{k,m}'$. 

Recall from \refS{sec:polya} that $U_i(m\ell)$  is the length of the $i$:th branch in $\cT^\circ_m$  (it can also be viewed as the number of balls of color $i$ at time $m\ell$ in the P\'olya urn model therein). Let $\law$ denote law.
\begin{lem}\label{lem:sumindi}
We have $\law\left(
\dgr\left(X_m,\cT'_{k,m}\right)\bigg\vert E\right) \eqdist \sum_{i=k+1}^m \lceil V_i U_i(m\ell) \rceil \cdot \I{W_{i}\le \frac{\Delta C_{i}}{C_{i}}}$; when $m=k$ the summation is $0$.
\end{lem}

\begin{proof}
Without loss of generality, assume that $m>k$. It follows from \refL{lem:unif} and the constructions of $\cT'_{k,m}$ and $\cT^\circ_{k,m}$ that 
\begin{equation}\label{eq:distance0}
\law\left(\dgr\left(X_m,\cT'_{k,m}\right)\big\vert E\right)\eqdist \law\left(\dgr\left(X^\circ_m,\cT^\circ_{k,m} \right)\big\vert E^\circ\right)
\end{equation}
Moreover, it follows from the construction of $\cT^\circ_{k,m}$ that
\begin{equation}
\label{eq:distance}
\law\left(\dgr\left(X^\circ_m,\cT^\circ_{k,m} \right)\big\vert E^\circ\right) = \sum_{i=k+1}^m M^\circ(S_i,m)\cdot \I{W_{i}\le \frac{\Delta C_{i}}{C_{i}}}.
\end{equation}
Now, due to the definition of $U_i(\cdot)$ and that the $X_i^\circ$ are placed uniformly according to normalized Lebesgue measure,
we have 
\[
\left(\I{W_{i}\le \frac{\Delta C_{i}}{C_{i}}} M^\circ(S_i,m)\right)_{i=k+1}^m \eqdist \left(\I{W_{i}\le \frac{\Delta C_{i}}{C_{i}}}\lceil V_i U_i(m\ell) \rceil\right)_{i=k+1}^m.
\]
Together with (\ref{eq:distance0}) and (\ref{eq:distance}) we may conclude the proof.
\end{proof}

\begin{proof}[{\bf Proof of \refL{lem:expbd}}]
This proof is an easy generalization of the argument in \citep[Section 1.2]{CH}. Note that $\lambda \cdot U_i(m\ell)\le 1$ for all $k+1\le i\le m$. Applying \refL{lem:sumindi} and using the bound $e^x \le 1+x+x^2$ for $0\le x\le 1$, we have for $\lambda>0$,
\begin{align*}
\E{\exp\left(\lambda \cdot \dgr(X_m,\cT'_{k,m}) \right)~\Big\vert~E,\cF_{k,n}} 
\le&~ 
\prod_{i=k+1}^m \left(\frac{C_{i-1}}{C_i} + \frac{C_i-C_{i-1}}{C_i}
\cdot \E{\exp(\lambda +\lambda \cdot  V_i U_i(m\ell) )}\right)\\
\le&~\prod_{i=k+1}^m \left( 1- \frac{\Delta C_{i}}{C_i} + \frac{\Delta C_i}{C_i}\cdot e^\lambda \cdot (1+ \lambda \cdot U_i(m\ell)) \right)\\
=&~ \prod_{i=k+1}^m \left(1+ \frac{\Delta C_i}{C_i}\cdot (e^\lambda-1) + \lambda e^\lambda \cdot \frac{\Delta C_i U_i(m\ell)}{C_i}\right).
\end{align*}
Notice that $0\le \lambda\le 1$, so $e^\lambda-1\le \lambda+\lambda^2\le 2\lambda \le 2\lambda U_i(m\ell)$ and $\lambda e^\lambda\le 3\lambda$. Thus, the above quantity is bounded by
\begin{align*}
\prod_{i=k+1}^m \left(1+ \frac{5\lambda \Delta C_i U_i(m\ell)}{C_i}\right)
\le \exp\left(5 \lambda \sum_{i=k+1}^m \frac{\Delta C_i U_i(m\ell)}{C_i}\right).
\end{align*}
Finally, for a tree $T$, write $e(T)$ for the set of edges of $T$. Given the event $E$, it follows from \refL{lem:frag} and the first assertion of \refL{lem:unif} that $X_m$ is on an edge uniformly chosen from $e(\cT'_m) \setminus e(\cT'_{k,m})$. Now, recall that $u$ is uniform over $v(\rT(n))\setminus v(\rT_k(n))$, where $\rT(n) =(\cT(n),\dgr)$. Since $n<m\ell+\ell$, it follows that
that
\[
\law\left(\dgr\left(u,\rT_k(n) \right)~\Big\vert~ \cF_{k,n}\right)
\stackrel{st}{\leq} \law\left( \ell + \dgr\left(X_m,\cT'_{k,m}\right)~\Big\vert~E,\cF_{k,n}\right),
\]
where~$\stackrel{st}{\leq}$ denotes stochastic domination. The lemma easily follows.
\end{proof}

\subsection{Moment bound for P\'olya urn}\label{sec:branch}

In this subsection, we prove \refL{lem:sum} under the framework of \refS{sec:polya}. Denote by $\bPclas{b}{w}{m}$ the distribution of white balls in a classical P\'olya urn after $m$ completed draws, starting with $b$ black and $w$ white balls. 
Denote by $\bPimm{\ell}{b}{w}{m}$ the number of white balls after $m$ completed steps in the P\'olya urn with immigration, starting with $b$ black and $w$ white balls: at the $n$th step, a ball is picked at random from the urn and returned along with an additional ball of the same color;
additionally, if~$n$ is a multiple of~$\ell$, then a black ball is added \emph{after} the~$n$:th draw and return. We use the notation $\law(\cdot)$ to denote the law of some random variable.

\begin{lem}{\em (\cite[Lemma 2.2 with $s=m_s=1$]{PRR}).}\label{lem:conditurns}
For all $k,n\in\N$ with $n\ge k\ell$,
\[
	M_k (n) \sim \bbPimm{\ell}{1}{(\ell+1)k}{n-k\ell }
\]
and
\begin{equation}\label{eq:ucondm}
\law\bklr{U_k(n)|M_k(n)}=\bPclas{(k-1)(\ell+1)}{1}{M_k(n)-(k-1)(\ell+1)-1}.
\end{equation}
\end{lem}

\begin{lem} \label{lem3}
Fix $k,n,q\in\N$ with $n\ge k\ell$. There is positive constant $c$ depending only on $q,\ell$ such that
\begin{equation}\label{10}
\E{ M_k(n)(M_k(n)+1)\cdots (M_k(n)+q(\ell+1)-1)}
\leq c k^q n^{q\ell}.
\end{equation}
\end{lem}

\begin{proof} 
From \cite[Lemma~4.1]{PRR}, for $Y\sim \bPimm{\ell}{1}{w}{t}$ and integer $q>0$, 
\be{
	\E{Y(Y+1)\cdots (Y+q(\ell+1)-1)}
	=\prod_{j=0}^{q(\ell+1)-1} (w+j) \prod_{i=0}^{t-1}\left(1+\frac{q(\ell+1)}{w+1+i+\lfloor i/\ell\rfloor}\right).
}
Setting $T = \floor{\frac{t-1}{\ell}}$, we calculate
\begin{align*}
&\E{Y(Y+1)\cdots (Y+q(\ell+1)-1)}\\
=&~\prod_{j=0}^{q(\ell+1)-1} (w+j) \prod_{i=0}^{t-1}\left(1+\frac{q(\ell+1)}{w+1+i+\lfloor i/\ell\rfloor}\right)\\
=&~\prod_{j=0}^{q(\ell+1)-1} (w+j) \prod_{r=0}^{T -1}\prod_{i=0}^{\ell-1} \left(1+\frac{q(\ell+1)}{w+1+i+r(\ell+1)}\right)
	 \prod_{p=\ell T}^{t-1}\left(1+\frac{q(\ell+1)}{w+1+p+T}\right)  \\
=&~w\prod_{j=0}^{q-2}(w+1+\ell+j(\ell+1)) \prod_{r=T}^{T+q-1} \prod_{i=0}^{\ell-1}(w+1+i +r(\ell+1))
	 \prod_{p=\ell T}^{t-1}\left(1+\frac{q(\ell+1)}{w+1+p+T}\right)  \\
\le&~(w+(q-1)(\ell+1))^{q}(w+q(\ell+1)+T(\ell+1))^{q\ell}
			\left(1+\frac{q(\ell+1)}{w+1+T(\ell+1)}\right)^\ell.
\end{align*}
Now setting~$w=(\ell+1)k$ and $t=n-k\ell $
and noting that with this choice of parameters,
\be{
w+T(\ell+1)\leq 1+(n-1)\frac{\ell+1}{\ell}+\ell,
}
we find
\begin{align*}
&\E{ M_k(n)(M_k(n)+1)\cdots (M_k(n)+q(\ell+1)-1)}   \\
\le&~ \left((\ell+1)(k-s+q)\right)^q\left(1+q(\ell+1)+\ell+n\frac{\ell+1}{\ell}\right)^{q\ell} \left(1+\frac{q(\ell+1)}{1+(n-1)\frac{\ell+1}{\ell}}\right)^\ell.\qedhere
\end{align*}
\end{proof}

\begin{lem}\label{lem:mkmombd}
Fix $k,n,q\in\N$ with $n\ge k\ell$. There is a constant $c=c(q,\ell)$ such that for all positive integer $j\leq q(\ell+1)$,
\be{
\E{ M_k(n)^j} \leq c k^{j/(\ell+1)} n^{j\ell/(\ell+1)}.
}
\end{lem}
\begin{proof}
For $j\leq q(\ell+1)$,
Jensen's (or H\"older's) inequality implies
\ben{\label{eq:jensenmom}
\E{ M_k(n)^j}\leq \left(\E{ M_k(n)^{q(\ell+1)}}\right)^{j/(q(\ell+1))}.
}
Using~\eq{10} of Lemma~\ref{lem3} now implies
\be{
\E{M_k(n)(M_k(n)+1)\cdots (M_k(n)+q(\ell+1)-1)}\leq c k^q n^{q \ell},
}
and the result for $j\le q(\ell+1)$ easily follows from this and~\eq{eq:jensenmom}.
\end{proof}

\begin{lem}\label{lem:urnmombd}
Fix $k,n,p\in\N$ with $n\ge k\ell$. There is a constant $c=c(p,\ell)$ such that
\[
\E{U_k(n)^p} \le c \left(\frac{n}{k}\right)^{p\ell/(\ell+1)}
\]
and 
\[
\E{\left(\frac{\Delta C_k U_k(n)}{C_k}\right)^{p}} \le c n^{p\ell/(\ell+1)} k^{-p(2\ell+1)/(\ell+1)}.
\]
\end{lem}

\begin{proof}
Recall from~\eq{eq:ucondm} of Lemma~\ref{lem:conditurns}
that
\be{
\law\bklr{U_k(n)|M_k(n)}=\bPclas{(k-1)(\ell+1)}{1}{M_k(n)-(k-1)(\ell+1)-1}.
}
Let the random variable~$B\sim\B[1, (k-1)(\ell+1)]$ be independent of $M_k(n)$. Conditional on $B$ and $M_k(n)$, let 
$X(M_k(n), B)$ be binomial with parameters $M_k(n)-(k-1)(\ell+1)-1$ and $B$.
By the de Finetti representation of
the classical P\'olya urn, we have
\ben{\label{eq:polyarep}
\law\bklr{U_k(n)|M_k(n)}=\law\bklr{(1+ X(M_k(n), B))|M_k(n)}.
}
H\"older's inequality implies that for non-negative $x,y$ 
and positive integer~$p$, $(x+y)^p\leq 2^{p-1} (x^p + y^p)$, and so starting from~\eq{eq:polyarep}, we have
\ben{\label{eq:ukmombd1}
\E{ U_k(n)^p  \big\vert M_k(n)}\leq 2^{p-1}\left(1+\E{  X(M_k(n),B)^p}\right).
}
 Now note that, if $\law(Y)=\Bi(N,q)$, 
then
for positive integer $p$,
and denoting Stirling numbers of the
second kind by ${p \brace j}$ (and note these are non-negative),
\be{
\E{ Y^p}=\sum_{j=0}^p {p \brace j} \E{ Y(Y-1)\cdots(Y-j+1)} \leq \sum_{j=0}^{p} {p\brace j} (Nq)^j.
}
So from~\eq{eq:ukmombd1}, condition on $M_k(n)$ (noting that $M_k(n)$
is independent of $B$) to find
\ben{\label{eq:ukmombd2}
\E{ U_k(n)^p \big\vert M_k(n)}\leq 2^{p-1}\left(1+  \sum_{j=0}^{p} {p\brace j} M_k(n)^j \E{ B^j} \right).
}
Standard formulas for beta moments imply
\ben{\label{eq:betamoms}
\E{ B^j}=\frac{\Gamma(j+1)\Gamma(1+(k-1)(\ell+1))}{\Gamma(1+(k-1)(\ell+1)+j)}\leq c k^{-j},
}
where $c=c(\ell,j)$ is a constant. Taking the expectation on both sides of (\ref{eq:ukmombd2}), together with \refL{lem:mkmombd} and (\ref{eq:betamoms}), yields that, for some $c=c(p,\ell)$,
\begin{align*}
\E{U_k(n)^p} \le 2^{p-1} \left(1 + \sum_{j=0}^p {p\brace j} \E{M_k(n)^j} \E{ B^j} \right)
\le c \sum_{j=0}^p k^{j/(\ell+1)-j} n^{j\ell/(\ell+1)}\le c \left(\frac{n}{k}\right)^{p\ell/(\ell+1)}.
\end{align*}

To deduce the last inequality, note that, under the notation of \refFt{fact:repr},
\[
\frac{\Delta C_k}{C_k} \eqdist 1- \left(\frac{E_1+\cdots+E_{k-1}}{E_1+\cdots + E_k}\right)^{\frac{1}{\ell+1}} \sim 1 - \mbox{Beta}((\ell+1)(k-1),1) \sim \mbox{Beta}(1,(\ell+1)(k-1)).
\]
(\ref{eq:betamoms}) then leads to $\E{\left(\frac{\Delta C_k}{C_k}\right)^{2p}} \le c k^{-2p}$. By Cauchy-Schwarz and the inequalities in the previous two displays,
\[
\E{\left(\frac{\Delta C_k U_k(n)}{C_k}\right)^p}
\le \left( \E{\left(\frac{\Delta C_k}{C_k}\right)^{2p}} \cdot \E{U_k(n)^{2p}}\right)^{1/2} \le c k^{-p} \left(\frac{n}{k}\right)^{p\ell/(\ell+1)}.\qedhere
\]
\end{proof}

\begin{proof}[{\bf Proof of \refL{lem:sum}}]
Recall that $\alpha=\ell/(\ell+1)$. For sufficiently large $N$, the event
\be{
\left\{\sum_{i=k(n)+1}^{\floor{\frac{n}{\ell}}+1}k(n)^{\alpha-\eps}\frac{\Delta C_i U_i(n)}{n^{\alpha}C_i}>1\right\}
}
is a subset of the union of events $\left\{\frac{\Delta C_i U_i(n)}{n^{\alpha}C_i}>c_\alpha i^{-\alpha-1+\eps/2}\right\}$, over $i=k(n)+1,\ldots, \floor{\frac{n}{\ell}}+1$ and for a sufficiently small constant $c_\alpha$.
This is because on the complement of this union, the sum is no greater than one. Next, we use Lemma~\ref{lem:urnmombd}, noting that $k(n)=\bigo(n^{\ell/(2\ell+1)})$, to find for integer $q>2/(\eps(\ell+1))$,
\ba{
\sum_{i=k(n)+1}^{\floor{\frac{n}{\ell}}+1}&\IP\left(\frac{\Delta C_i U_i(n)}{n^\alpha C_i}>c_\alpha i^{-\alpha-1+\eps/2}\right)
\leq \sum_{i=k(n)+1}^{\floor{\frac{n}{\ell}}+1}\frac{\E{\left(\frac{\Delta C_i U_i(n)}{n^{\alpha}C_i}\right)^{q(\ell+1)}}}{c_\alpha^{q(\ell+1)} i^{(-\alpha-1+\eps/2)q(\ell+1)}} \\
	&\qquad \leq c\sum_{i=k(n)+1}^{\floor{\frac{n}{\ell}}+1} i^{(\alpha+1-\eps/2)q(\ell+1)-q(2\ell+1)}=\bigo\left(k(n)^{-\eps q (\ell+1)/2+1}\right),
}
where $c$ is a constant depending only on $q,\ell$.
Multiplying by $n$ both sides of the above inequalities and then summing over $n\in\N$ yields
\begin{equation}
\label{eq:sumbd1}
\sum_{n\in\N} n\cdot  \IP\left(k(n)^{\alpha-\eps}\sum_{i=k(n)+1}^{\floor{\frac{n}{\ell}}+1}\frac{\Delta C_i U_i(n)}{n^{\alpha}C_i}>1\right) = O\left(\sum_{n\in\N}n^2\cdot k(n)^{-\veps q(\ell+1)/2 +1}\right).
\end{equation}
Since $k(n) = \Omega(n^{1/100})$, we can choose~$q$ large enough
so that $\sum_{n\in\N}n^2\cdot k(n)^{-\veps q(\ell+1)/2 +1}$ is finite. 

Furthermore, using a union bound, together with Markov's inequality and \refL{lem:urnmombd}, 
\begin{align*}
\p{\max_{k(n)+1\le i\le \lfloor \frac{n-1}{\ell}\rfloor+1} U_i(n) > n^\alpha k(n)^{-\alpha+\veps}} \le &~ \sum_{i=k(n)+1}^{\lfloor \frac{n}{\ell}\rfloor+1} \p{U_i(n)> n^\alpha k(n)^{\alpha+\veps}}\\
\le&~ \sum_{i=k(n)+1}^{\lfloor \frac{n}{\ell}\rfloor+1} \frac{\E{U_i(n)^p}}{n^{p\alpha} k(n)^{-p\alpha+p\veps}}\le a k(n)^{1-p\veps},
\end{align*}
for any $p\in\N$ and some $a=a(p,\ell)$. Now take $p$ large enough to get
\begin{equation}
\label{eq:sumbd2}
\sum_{n\in\N} n\cdot  \p{\max_{k(n)+1\le i\le \lfloor n/\ell\rfloor+1} U_i(n) > n^\alpha k(n)^{-\alpha+\veps}}<\infty.
\end{equation}
The lemma then follows from (\ref{eq:sumbd1}), (\ref{eq:sumbd2}), and the triangle inequality.
\end{proof}

\subsection{Convergence of measures}\label{sec:prok}

In this subsection, we prove \refL{lem:prok}. Fix integers $k,n$ with $n\ge (k-1)\ell$ unless specified otherwise. Recall from \refS{sec:intro} that $v_i$ denotes the vertex inserted in a uniform edge of the combinatorial tree $\rT(i-1)$, and, if $\ell$ divides $i$, $v_i$ is a branchpoint (i.e., there is a new edge attached to $v_i$ at the time $v_i$ appears). Note that $\{v_0,v_1,\ldots, v_{k\ell},L_1,\ldots,L_k\} \subset v(\rT_k(n))$, where $v_0$ is the root, $L_1,\ldots,L_k$ are the first $k$ leaves. Hereafter, conditioned on $|v(\rT_k(n))| = (\ell+1)k + m+1$ for some appropriate integer $m\coloneqq m(\rT_k(n))$, list the internal vertices of $\rT_k(n)$ as $(v_1,\ldots,v_{k\ell}, v_{i_1},\ldots,v_{i_m})$, in the order of appearance; the other $k+1$ vertices of $\rT_k(n)$ are the leaves and the root. For convenience, denote $i_0 = k\ell$.

Given $w\in v(\rT_k(n))$, recall that $\overline{w}$ is the maximal subset of $v(\rT(n))$ such that the removal of $w$ disconnects $\overline{w}$ from $\rT_k(n)$; so for $1\le i\le k\ell-1$, $\overline{v_i} = \emptyset$.  For all integer $0\le j \le m$, let $T_{i_j}$ be the subtree of $\rT(n)$ restricted to $\overline{v_{i_j}}\cup \{v_{i_j}\}$, and let 
\[
n_{i_j} = |v(T_{i_j})|.
\]
 If $\ell$ divides $i_j$, then $\overline{v_{i_j}}\neq \emptyset$ and $n_{i_j}>1$, otherwise, $\overline{v_{i_j}}=\emptyset$ and $n_{i_j}=1$.

Next, list the vertices $v_{i_0},v_{i_1},\ldots,v_{i_m}$ in the breadth-first search order of $\rT_k(n)$, as $w_0,w_1,\ldots,w_m$. So there is a bijection $f:\{i_0,\ldots,i_m\}\to\{0,\ldots,m\}$ such that, in $\rT_k(n)$, $v_{i_j}$ is identified with $w_{f(i_j)}$ and $T_{i_j}$ is attached to $w_{f(i_j)}$. Now, let $\sigma\coloneqq (\sigma(i_j):0\le j\le m)$ be a uniformly chosen random permutation of $i_0,\ldots,i_m$. We construct a random tree $\rT(n,\sigma)$ by identifying the vertex $w_j$ of $\rT_k(n)$ with the vertex $v_{\sigma(i_j)}$ of $T_{\sigma(i_j)}$, for each $0\le j\le m$. Since each $v_{i_j}$ is inserted in a uniform edge of the existing tree $\rT_k(i_j-1)$, uniformly permuting the attaching locations of $T_{i_j}$ does not change the law of the resulting tree. It thus follows that $\rT(n,\sigma) \eqdist \rT(n)$. Write $|\bn_{m,k,n}|_2 = (\sum_{j=0}^m n_{i_j}^2)^{1/2}$ and $N_n=|v(\rT(n))|=n+\lfloor n/\ell\rfloor +2$. Recall that $m\coloneqq |v(\rT_k(n))| - (\ell+1)k-1$ is the number of the internal vertices of $\rT_k(n)$ except for $\{v_1,\ldots,v_{k\ell}\}$.

\begin{lem}\label{lem:abw}
Let $k:\N\to\N$ be such that $k(n) = \Omega((\log n)^{10})$ and $k(n)=o(n^{1/2})$. For all $n\in\N$, denote the event
\[
F_n = \left\{\left\vert C_{k(n)} - k(n)^{\frac{1}{\ell+1}}\right\vert < 10 k(n)^{\frac{1}{\ell+1}-\frac14} \right\}\bigcap \left\{\left\vert \frac{c}{n^\alpha}\cdot |v(\rT_{k(n)}(n))| - C_{k(n)}\right\vert<1\right\}.
\]
 Fix a sufficiently large $n$ and write $k=k(n)$, $m=m(n)=|v(\rT_k(n))| - (\ell+1)k-1$. Then for any $V\subset v(\rT_k(n))$ and $t>4c(\ell+4)\cdot n^{\alpha}k^{\frac{1}{\ell+1}}$,
\begin{equation*}
\p{\left\vert \overline{\nu}_{k,n}(V) - \nu_{k,n}(V) \right\vert > \frac{2t}{N_n}~\bigg\vert~ F_n}\le 2 \E{\exp\left(-\frac{2t^2}{|\bn_{m,k,n}|_2^2}\right)}.
\end{equation*}
\end{lem}

\begin{proof}
For the duration of the proof, let $V\subset v(\rT_k(n))$, and write $V'= V\bigcap\{v_{i_0},\ldots,v_{i_m}\}$, $N=N_n=|v(\rT(n))|$. Recall that $\sigma=(\sigma(i_j):0\le j\le m)$ is a uniform permutation of $i_0,\ldots,i_m$. Write $\sigma_j = \sigma(i_j)$ for each $0\le j\le m$. By definition, 
\begin{equation}\label{eq:meas}
\overline{\nu}_{k,n}(V)=\frac{\sum\limits_{v_{i_j}\in V'}n_{i_j} + |V\setminus V'|}{N} \eqdist \frac{\sum\limits_{v_{\sigma_j}\in V'}n_{\sigma_j} + |V\setminus V'|}{N} \coloneqq \frac{\sum\limits_{v_{\sigma_j}\in V'}|v(T_{\sigma_j})|+ |V\setminus V'|}{N} ;
\end{equation}
the second equality follows from the fact that $\rT(n,\sigma)\eqdist \rT(n)$, so $\sum_{v_{i_j}\in V'} n_{i_j} \eqdist \sum_{v_{\sigma_j}\in V'} n_{\sigma_j}$. Note that, exchangeability resulting from the uniform permutation $\sigma$ only exhibits through $(n_{\sigma_j}:0\le j\le m)$, but not $V\setminus V'$. Let $n'= \sum_{j=0}^m n_{i_j} = N - (\ell+1)k$. To use exchangeability to deduce the tail bound for $\left\vert\overline{\nu}_{k,n}(V)-\nu_{k,n}(V)\right\vert$, we first show that it is close to $\left\vert\frac{\sum_{v_{\sigma_j}\in V'} n_{\sigma_j}}{n'} - \frac{|V'|}{m+1}\right\vert$, and then employ exchangeability to bound the latter. Indeed, by the triangle inequality, writing $M=M_n=|v(\rT_k(n))|$ and noting that $\nu_{k,n}(V) = \frac{|V|}{M}$, we have
\begin{align*}
\left\vert\overline{\nu}_{k,n}(V)-\nu_{k,n}(V)\right\vert
\eqdist &~ \left\vert \frac{\sum_{v_{\sigma_j}\in V'}n_{\sigma_j} + |V\setminus V'|}{N} - \frac{|V|}{M}\right\vert\\
\le&~ \frac{n'}{N}\cdot 
 \left\vert \frac{\sum_{v_{\sigma_j}\in V'}n_{\sigma_j}}{n'} - \frac{|V'|}{m+1}\right\vert
 +\left\vert \frac{|V\setminus V'|}{N}-\frac{|V|}{M} + \frac{|V'|}{m+1}\cdot \frac{n'}{N}\right\vert.
\end{align*}
Hence, even conditionally,
\begin{align}
& \p{\Big\vert\overline{\nu}_{k,n}(V)-\nu_{k,n}(V)\Big\vert> \frac{2t}{N}~\bigg\vert~F_n}\notag \\
\le&~  \p{\Big\vert\sum_{v_{\sigma_j}\in V'}\frac{n_{\sigma_j}}{n'}- \frac{|V'|}{m+1}\Big\vert> \frac{t}{n'}~\bigg\vert~F_n}\label{in:prokbd1} \\
&\instep + \p{\Big\vert \frac{|V\setminus V'|}{N}-\frac{|V|}{M} + \frac{|V'|}{m+1}\cdot \frac{n'}{N}\Big\vert> \frac{t}{N}~\bigg\vert~F_n}.\label{in:prokbd}
\end{align}

We now compute (\ref{in:prokbd}). It is convenient to keep in mind that we have chosen $n$ sufficiently large, and the variables $k=k(n)$, $M=M_n= |v(\rT_k(n))|$, $m=m(n)=M - (\ell+1)k-1$, $N=N_n = |v(\rT(n))|$, depend on $n$. Note that $|V\setminus V'|\le (\ell+1)k+1$ and $t>4c(\ell+4)\cdot n^{\alpha}k^{\frac{1}{\ell+1}}$, so $t>2|V\setminus V'|$ for large enough $n$. It follows that 
\begin{align}
\left\{\left\vert \frac{|V\setminus V'|}{N}-\frac{|V|}{M} + \frac{|V'|}{m+1}\cdot \frac{n'}{N}\right\vert > \frac{t}{N}\right\} 
\subset&~ \left\{\frac{|V\setminus V'|}{N}>\frac{t}{2N}\right\} \bigcup 
\left\{\left\vert \frac{|V|}{M} - \frac{|V'|}{m+1}\cdot \frac{n'}{N} \right\vert >\frac{t}{2N}\right\} \notag\\
=&~ \emptyset \bigcup \left\{\left\vert |V| \cdot \frac{N}{M} - |V'| \cdot \frac{n'}{m+1}\right\vert > \frac{t}{2} \right\}.\label{eq:setrela}
\end{align}
Given the event $F_n$, it is easily seen that $\frac{1}{2c}\cdot  n^\alpha k^{\frac{1}{\ell+1}}\le M\le \frac{2}{c}\cdot n^\alpha k^{\frac{1}{\ell+1}}$. Also, we have $N=N_n =\Theta( n)$ and $n'=  N -  (\ell+1)k$, so 
\[
\frac{N}{M}\leq \frac{n'}{m+1} = \frac{N-(\ell+1)k}{M- (\ell+1)k} \le   \frac{N}{M} + \frac{(\ell+1)k n}{M(M-(\ell+1)k)}.
\]
Note also that $|V\setminus V'|\le (\ell+1)k+1$, $|V'|\le M-(\ell+1)k$, and $N \le n(1+1/\ell)+2$, it then follows by the triangle inequality that on the event $F_n$,
\begin{align*}
\left\vert |V| \cdot \frac{N}{M} - |V'| \cdot \frac{n'}{m+1}\right\vert 
\le&~ \left\vert |V|\cdot \frac{N}{M} - |V'| \cdot \frac{N}{M}\right\vert + |V'|\cdot\frac{(\ell+1)k n}{M(M-(\ell+1)k)}\\
=&~ |V\setminus V'| \cdot \frac{N}{M} + |V'|\cdot\frac{(\ell+1)k n}{M(M-(\ell+1)k)}\\
\le&~ \left\{(\ell+1)k+1\right\}\cdot \frac{N}{M} +  \frac{(\ell+1) k n}{M}\\
\le&~ (2\ell+3) \cdot \frac{kn}{M} \le 2c(2\ell+3) \cdot n^{1-\alpha} k^{\frac{\ell}{\ell+1}},
\end{align*}
where the last inequality is due to $M\ge \frac{1}{2c}\cdot n^\alpha k^{\frac{1}{\ell+1}}$ on $F_n$. Since $t> 4c(\ell+4)\cdot n^\alpha k^{\frac{1}{\ell+1}}$, we thus have
\[
\left\{\left\vert |V| \cdot \frac{N}{M} - |V'| \cdot \frac{n'}{m+1}\right\vert > \frac{t}{2}~\right\}\subset \left\{ 2c(2\ell+3) \cdot n^{1-\alpha} k^{\frac{\ell}{\ell+1}}>2c(\ell+4)\cdot n^\alpha k^{\frac{1}{\ell+1}}\right\}=\emptyset, 
\]
where the last equality is due to the fact that $\alpha\ge \frac12$, and when $\alpha=\frac12$ we have $\ell=1$. Combined with (\ref{eq:setrela}), we have
\begin{equation}
\label{in:prokbd2}
 \p{\Big\vert \frac{|V\setminus V'|}{N}-\frac{|V|}{M} + \frac{|V'|}{m+1}\cdot \frac{n'}{N}\Big\vert> \frac{t}{N}~\bigg\vert~F_n}=0.
\end{equation}

It remains to bound (\ref{in:prokbd1}), shown below. The rest of the proof follows a similar argument as in \citep[Lemma 5.3]{ABW}, so we only point out the differences, and refer the reader to that work for omitted explanations. Let $r_0,\ldots,r_m$ be independent random variables with uniform law over $\{i_0,\ldots,i_m\}$. Recall that $n'= \sum_{j=0}^m n_{i_j}$. It follows by symmetry that
\[
\E{\sum_{v_{r_j}\in V'} n_{i_j}~\bigg\vert~n'}= \E{\sum_{j=0}^m n_{i_j} \cdot \I{v_{r_j}\in V'}~\bigg\vert~n'} = n' \cdot \p{v_{r_1}\in V'} = n'\cdot \frac{|V'|}{m+1}.
\]
Taking $a\coloneq\frac{4t}{|\bn_{m,k,n}|_2^2}$ and applying Markov's inequality as in \citep[Theorem 2.5]{Mc} gives a Hoeffding-type inequality for $\sum_{v_{\sigma_j}\in V'}n_{\sigma_j}$: for any $t>0$,
\begin{align}
 \p{\Big\vert\sum_{v_{\sigma_j}\in V'}\frac{n_{\sigma_j}}{n'}- \frac{|V'|}{m+1}\Big\vert> \frac{t}{n'}~\bigg\vert~F_n} 
=&~ \p{\Big\vert\sum_{v_{\sigma_j}\in V'} n_{\sigma_j} - n'\cdot \frac{|V'|}{m+1}\Big\vert> t~\bigg\vert~F_n} \notag\\
\le&~ e^{-at}\cdot \E{\exp\Big(a\cdot \Big\vert\sum_{v_{\sigma_j}\in V'} n_{\sigma_j} - n'\cdot \frac{|V'|}{m+1}\Big\vert\Big) ~\bigg\vert~F_n}\notag\\
\le&~ e^{-at}\cdot \E{\exp\Big(a\cdot \Big\vert\sum_{v_{r_j}\in V'} n_{r_j} -n'\cdot \frac{|V'|}{m+1}\Big\vert\Big) ~\bigg\vert~F_n}\notag\\
\le  &~2\E{\exp\pran{-\frac{2t^2}{|\bn_{m,k,n}|_2^2}}};\notag
\end{align}
the first inequality follows from Markov's inequality; the second one is due to \citep[Proposition 20.6]{Ald} and \citep[Theorem 2]{Str}; the last inequality follows from a straightforward calculation as in \citep[Lemma 2.6]{Mc}. Together with (\ref{in:prokbd1}) and (\ref{in:prokbd2}), we may conclude the proof.
\end{proof}

 Given a graph $G$ and a subgraph $G'$ of $G$, write $G-G'$ for the components obtained by removing all edges and vertices of $G'$ from $G$. For all integers $k,n$ with $n\ge(k-1)\ell$, let
\begin{equation}\label{eq:Skn}
S_{k,n}= S\left(\rT_{k}(n)\right) =\max\left\{|v(T)|: T \in \rT(n) - \rT_k(n)\right\}.
\end{equation}

\begin{lem}\label{lem:maxsize}
Let $k:\N\to\N$ be such that $k(n)=\Omega\left(n^{1/100}\right)$ and $k(n) = o(n)$ and restrict $\ell\geq2$. There exists $a=a(\ell)>0$ such that, for sufficiently large $n$,
\[
\p{S_{k(n),n} \ge n\cdot k(n)^{-\frac{8}{3(\ell+1)}}} \le a\cdot n^{-2}.
\] 
\end{lem}

The proof of \refL{lem:maxsize} is deferred to \refS{sec:maxsize}.

\begin{proof}
[{\bf Proof of \refL{lem:prok}}]
For all $n\in\N$, let $\veps_n= k(n)^{-\frac{1}{12(\ell+1)}}$ and 
\[
M_n= \left\lceil \veps_n^{-1}\cdot \frac{c}{n^\alpha}\cdot|v(\rT_{k(n)}(n))| \right\rceil. 
\]
In the remaining proof we fix a large $n\in\N$ and write $k=k(n)$, unless we consider varying $n$.  With a similar argument as in \refP{prop:finite}, we find a covering, denoted by $\{B_{n,1},\ldots,B_{n,M_n}\}$, of $(v(\rT_{k}(n)),\frac{c}{n^\alpha}\cdot \dgr)$, with diameter at most $\veps_n$. Let $A_{n,1} = B_{n,1}$, and for $i>1$, let $A_{n,i} = B_{n,i}\setminus A_{n,i-1}$. Then $\{A_{n,1},\ldots,A_{n,M_n}\}$ forms a disjoint cover of $(v(\rT_{k}(n)),\frac{c}{n^\alpha}\cdot \dgr)$, with diameter at most $\veps_n$.

This paragraph follows a similar argument as in \cite[Corollary 6.2]{ABW}, and we refer the reader to that work for omitted details. Recall that $d_{k,n}$ denotes the L\'evy-Prokhorov distance on $(v(\rT_{k}(n)),\frac{c}{n^\alpha}\cdot \dgr)$. We claim that,
\begin{align*}
&\left\{ d_{k,n}\left(\overline{\nu}_{k,n}, \nu_{k,n}\right) >\veps_n\right\}
\subset \left\{|\overline{\nu}_{k,n}(A_{n,j}) - \nu_{k,n}(A_{n,j})|>\frac{\veps_n}{M_n}~\mbox{ for some } 1\le j\le M_n\right\}; 
\end{align*}
a quick proof is provided as follows. Suppose that $d_{k,n}\left( \overline{\nu}_{k,n}, \nu_{k,n}\right)>\veps_n$. Then there exists a set $S\subset v(\rT_{k}(n))$ such that either $\overline{\nu}_{k,n}(S^{\veps_n}) < \nu_{k,n}(S) - \veps_n$ or $\nu_{k,n}(S^{\veps_n}) < \overline{\nu}_{k,n}(S) - \veps_n$. Since $\{A_{n,1},\ldots,A_{n,M_n}\}$ is a disjoint cover of $\rT_{k}(n)$, there exists $j$ such that either
\[
\overline{\nu}_{k,n}(S^{\veps_n}\cap A_{n,j}) < \nu_{k,n}(S\cap A_{n,j}) - \veps_n/M_n \mbox{ or}
\]
\[
\nu_{k,n}(S^{\veps_n}\cap A_{n,j}) < \overline{\nu}_{k,n}(S\cap A_{n,j}) - \veps_n/M_n.
\]
So $S\cap A_{n,j} \neq \emptyset$. Since the diameter of $A_{n,j}$ is at most $\veps_n$, we have $A_{n,j}\subset S^{\veps_n}$. So, either
\[
\overline{\nu}_{k,n}(A_{n,j}) < \nu_{k,n}(A_{n,j}) - \veps_n/M_n \mbox{ or }\nu_{k,n}(A_{n,j}) < \overline{\nu}_{k,n}(A_{n,j}) - \veps_n/M_n.
\]
Hence the claim. 

Next, let as before
\[
F_n = \left\{\left\vert C_{k(n)} - k(n)^{\frac{1}{\ell+1}}\right\vert < 10 k(n)^{\frac{1}{\ell+1}-\frac14} \right\}\bigcap \left\{\left\vert \frac{c}{n^\alpha}\cdot |v(\rT_{k(n)}(n))| - C_{k(n)}\right\vert<1\right\},
\]
and note that it follows from \refL{lem:bd} and \refL{lem:convas} that, for sufficiently large $n$, 
\begin{equation}\label{in:smallprob}
\p{F_n}>1- 2 e^{-k(n)^{1/2}/4} - 2\exp\left(-\frac{1\cdot n^\alpha}{32 c k(n)^{\frac{1}{\ell+1}}}\right)-e^{-k(n)^{1/3}} > 1-2 e^{-k(n)^{1/3}}.
\end{equation}
We now easily obtain that
\begin{align}
&\p{d_{k,n}\left( \overline{\nu}_{k,n}, \nu_{k,n}\right) >\veps_n,~F_n} \notag\\
&\qquad\le~  \p{|\overline{\nu}_{k,n}(A_{n,j}) - \nu_{k,n}(A_{n,j})|>\frac{\veps_n}{M_n}~\mbox{ for some } 1\le j\le M_n,~F_n}.  \label{in:measbd1}
\end{align}
On the event $F_n$, we have 
\[
M_n \le m_n \coloneqq \left\lceil \veps_n^{-1}\left\{1+k(n)^{\frac{1}{\ell+1}}\left(1+10k(n)^{-\frac14}\right)\right\}\right\rceil =\Theta\left(\veps_n^{-1} k(n)^{\frac{1}{\ell+1}}\right).
\] 
Furthermore, let $N_n =|v(\rT(n))|=n+\lfloor n/\ell\rfloor+2$, $m =m(n)= |v(\rT_{k}(n))| - (\ell+1)k-1$, and write $|\bn_n|_2 = \left(\sum_{j=0}^{m} n_{i_j}^2\right)^{1/2}$; $\bn_n$ here is $\bn_{m,k,n}$ in \refL{lem:abw}. Set $t= \frac{\veps_n N_n}{2m_n}$ and note that since $k(n) = o(n^{1/3})$, it is easily checked that $t= \Theta\left(n\cdot k(n)^{-\frac{7}{6(\ell+1)}}\right)>4c(\ell+1)n^\alpha k^{1/(\ell+1)}$ for large~$n$ and so satisfies the assumption in \refL{lem:abw}. Applying the bound on $M_n$ and \refL{lem:abw} to (\ref{in:measbd1}), it follows that
\begin{align}
\p{d_{k(n),n}\left( \overline{\nu}_{k,n}, \nu_{k,n}\right) >\veps_n~\Big\vert~F_n}
\le m_n \cdot 2 \E{\exp\left( -\frac{2\veps_n^2 N_n^2}{|\bn_n|_2^2\cdot 4 m_n^2}\right)} .\label{in:prokexpbd}
\end{align}
Now bound
\[
|\bn_n|_2^2\leq \left(\sum_{j=0}^{m(n)} n_{i_j}\right) \max_{0\le j\le m(n)} n_{i_j}\leq n \cdot S_{k(n),n},
\]
where 
$S_{k(n),n} = \max_{0\le j\le m(n)} n_{i_j}+1$. Thus,
\begin{equation}\label{eq:prokexpbd}
\frac{\veps_n^2N_n^2}{|\bn_n|_2^2m_n^2} = \Omega\left(\frac{\veps_n^2 n}{S_{k(n),n} m_n^2}\right) = \Omega\left(\frac{n\cdot k(n)^{-\frac{7}{3(\ell+1)}}}{S_{k(n),n}} \right).
\end{equation}
Since $F_n$ and $S_{k(n),n}$ are independent, it then follows from the triangle inequality, (\ref{in:prokexpbd}), (\ref{eq:prokexpbd}), and \refL{lem:maxsize} that
\begin{align*}
&\p{d_{k(n),n}\left(\overline{\nu}_{k(n),n},\nu_{k(n),n}\right) > \veps_n,~F_n}\\
\le&~\p{d_{k(n),n}\left(\overline{\nu}_{k(n),n},\nu_{k(n),n}\right) > \veps_n, \,\,  S_{k(n),n}< n\cdot k(n)^{-\frac{8}{3(\ell+1)}},~F_n} + \p{S_{k(n),n}\ge n\cdot k(n)^{-\frac{8}{3(\ell+1)}}} \\ 
=&~O\left( m_n \cdot \exp\left(-\frac{n\cdot k(n)^{-\frac{7}{3(\ell+1)}}}{n\cdot k(n)^{-\frac{8}{3(\ell+1)}}}\right)\right) + O\left(n^{-2} \right)\\
=&~O\left( m_n\cdot \exp\left(-k(n)^{\frac{1}{3(\ell+1)}} \right) \right) + O\left(n^{-2}\right).
\end{align*}
Together with (\ref{in:smallprob}), we further deduce
\begin{align*}
\p{d_{k(n),n} \left(\overline{\nu}_{k(n),n},\nu_{k(n),n}\right) > \veps_n}
\le&~  \p{d_{k(n),n} \left(\overline{\nu}_{k(n),n},\nu_{k(n),n}\right) > \veps_n,~F_n} + \p{F_n^c}\\
=&~ O\left( m_n \cdot \exp\left(- k(n)^{\frac{1}{3(\ell+1)}}\right) + n^{-2} + e^{-k(n)^{1/3}}\right).
\end{align*}
Finally, recalling that $m_n = \Theta\left(k(n)^{\frac{13}{12(\ell+1)}}\right)$ and $k(n)= \Omega\left(n^{1/100}\right)$, summing over $n$ and applying the Borel-Cantelli lemma yields the almost sure convergence.
\end{proof}

\subsection{Maximal occupation of an infinite-colors P\'olya urn}\label{sec:maxsize}

In this subsection we prove \refL{lem:maxsize}, by viewing sizes of the subtrees as the occupations of a modified version of infinite-colors P\'olya urn, introduced below.

At time $0$, the urn contains $b\in\N$ black balls and $w\in\N$ balls of color $1$. At time $t\in\N$, a ball is chosen at random from the urn and returned along with an additional ball of the same color. Additionally, if $\ell$ divides $t$ and the urn has black balls and balls of color $\{1,\ldots,p\}$, then (i) if the chosen ball is black, add a ball of color $p+1$; or (ii) if the chosen ball is non-black, add a ball of the same color as the chosen ball. For each $i\in\N$, let $\cU_i(t;b,w)$ be the number of color-$i$ balls and let $\cU_0(t;b,w)$ be the number of black balls in the urn after $t\in\N$ draws. Let $\cM_i(t;b,w) = \cU_1(t;b,w)+\cdots+\cU_i(t;b,w)$, noticing that the sum does not include $\cU_0(t;b,w)$. Let $\cS_i(b,w)$ be the random time that color $i$ appears.

Recalling the definition~\eqref{eq:Skn}, the key relation between this urn model and the quantities appearing in Lemma~\ref{lem:maxsize} is
\[
S_{k,n} \coloneqq \max\left\{ |v(T)|: T\in \rT(n) -\rT_k(n)\right\}
\eqdist \max_{i\in\N}\cU_i\left(n - k\ell; (\ell+1)k,1\right).
\] 
We prove the lemma by deducing moment bounds of $S_{k,n}$ from moment bounds of the $\cU_i$. 

For the moment-bounds on $\cU_i(t;b,w)$ (and $\cM_i(t;b,w)$), we consider the following auxiliary $2$-colors P\'olya urn, which is equivalent to the above urn by regarding color $1$ as white and $\{2,3,\ldots\}$-colors as black. At time $0$, the urn contains $b\in\N$ black balls and $w\in\N$ white balls. At time $t\in\N$, a ball is chosen at random from the urn and returned along with an additional ball of the same color. Additionally, if $\ell$ divides $t$, then an additional ball of the chosen color is added. Let $W(t;b,w)$ be the number of white balls in the urn after $t$ draws.
 
\begin{fact}\label{fact:urndist}
Fix $b,w,i\in\N$. Write $\cM_j(\cdot)=\cM_j(\cdot ;b,w)$, $\cU_j(\cdot)=\cU_j(\cdot ;b,w)$, and $\cS_j = \cS_j(b,w)$, for all $j\in\N$. Given $\cS_i$, for integer $t>\cS_i$, 
\begin{equation}\label{eq:mdist}
\E{\cM_i(t)\big\vert \cS_i,\cU_0(\cS_i),\cM_i(\cS_i)} \eqdist W\left( t-\cS_i; \cU_0(\cS_i) , \cM_i(\cS_i) \right)
\end{equation}
and
\begin{equation}\label{eq:ucondmdist}
\E{\cU_i(t)\big\vert \cM_i(t),\cM_i(\cS_i)} \eqdist W\left(\cM_i(t)-\cM_i(\cS_i); \cM_i(\cS_i)-1,1 \right).
\end{equation}
\end{fact}

\refFt{fact:urndist} can be shown similarly as for \refL{lem:conditurns} so we omit the proof.

\begin{lem}\label{lem:momentbd}
Fix $b,w\in\N$. For all non-negative integers $t$ and $p$, there exists $f_{p,\ell}>0$ not depending on $b,w$ such that
\begin{equation}\label{in:momentbd}
\E{W(t;b,w)^p}\le f_{p,\ell} w^p\cdot \left(1+\frac{t}{\alpha(b+w)}\right)^p\cdot \left\{1+ \log\left(1+\frac{t}{\alpha(b+w)}\right)\right\}\eqqcolon d_{t,p}.
\end{equation}
\end{lem}

\begin{proof}
Write $n_j = b+w  + j + \lfloor j/\ell\rfloor$ and $W_{j} = W(j; b,w)$ for all non-negative integer $j$. Since for all $t\in\N$,
\[
\E{W_t\big\vert W_{t-1}} = \frac{W_{t-1}}{n_{t-1}}\cdot \left(W_{t-1}+1+\I{\ell | t}\right)+ \frac{n_{t-1}-W_{t-1}}{n_{t-1}}\cdot W_{t-1},
\]
it follows that, for all $p\in\N$ and some positive constants~$e_p',~e_p$ only depending on~$p$ (noting that $W_{t}\geq 1$ for all~$t$),
\begin{align}
&\E{W_t^p\big\vert W_{t-1}}\notag\\
&=~ 
\frac{W_{t-1}}{n_{t-1}}\cdot \left\{(W_{t-1}+1)^p\cdot \I{\ell \nmid t} + (W_{t-1}+2)^p\cdot \I{\ell \vert t} \right\}  + \frac{n_{t-1} - W_{t-1}}{n_{t-1}}\cdot W_{t-1}^p\notag\\
&=~ \frac{W_{t-1}}{n_{t-1}}\cdot \left[ (W_{t-1}+1)^p + \I{\ell | t}\cdot \left\{ (W_{t-1}+2)^p - (W_{t-1}+1)^p\right\}\right] + \frac{n_{t-1}- W_{t-1}}{n_{t-1}}\cdot W_{t-1}^p\notag\\
&\le ~ W_{t-1}^p+ \frac{W_{t-1}^p}{n_{t-1}}\cdot \left\{ (W_{t-1}+1)(1+1/W_{t-1})^{p-1} + p\cdot \I{\ell | t}\cdot (1+2/W_{t-1})^{p-1}  - W_{t-1}\right\}\notag\\
&\le~ W_{t-1}^p+ \frac{ W_{t-1}^p}{n_{t-1}}\cdot \left\{
(W_{t-1}+1) \left(1 + \frac{p-1}{W_{t-1}} + \frac{e_{p}'}{W_{t-1}^2}\right) + p\cdot \I{\ell\vert t}\cdot \left(1+\frac{e_{p}'}{W_{t-1}}\right) - W_{t-1}\right\}\notag\\
&\le~ W_{t-1}^p \cdot \left( 1+ \frac{p(1+\I{\ell\vert t}) }{n_{t-1}}
\right) + W_{t-1}^{p-1}\cdot \frac{e_{p}}{n_{t-1}}.\label{in:pbd}
 \end{align}

Next, we use induction on $p$ to prove the bound (\ref{in:momentbd}). Clearly (\ref{in:momentbd}) is true for $p=0$. Now assume that it holds for $p-1$ where $p\in\N$ and for all non-negative integer $t$. We are to show that it also holds for $p$. Averaging (\ref{in:pbd}) over $W_{t-1}$ and using the induction hypothesis yields that, for all $t\in\N$,
\begin{align*}
\E{W_t^p}\le \E{W_{t-1}^p} \cdot\left(1+\frac{p(1+\I{\ell|t})}{n_{t-1}}\right)
+ d_{t-1,p-1}\cdot \frac{e_{p}}{n_{t-1}},
\end{align*}
recalling the definition of $d_{t-1,p-1}$ from (\ref{in:momentbd}). For $m\le t$, let
\be{
\gamma_{m,t-1}=\prod_{j=m}^{t-1} \left(1 + \frac{p(1+\I{\ell\vert (j+1)})}{n_j}\right),
}
where $\prod_{j=t}^{t-1}(\cdot)\coloneqq 1$. Applying the inequality above recursively, we find
\begin{align}
\E{W_t^p} \le&~  w^p\gamma_{0,t-1} + \sum_{i=0}^{t-1} \frac{ d_{i,p-1} e_{p}}{n_i} \gamma_{i+1,t-1}.\label{in:p}
\end{align}
To bound and simplify $\gamma_{m,t-1}$, write $t'= \lfloor \frac{t}{\ell}\rfloor$, $m'=\floor{\frac{m}{\ell}}$ and using the inequality $1+x\le e^x$ for all $x\in\R$,
\begin{align*}
\gamma_{m,t-1}&\leq~\prod_{j=\ell  m' }^{t-1} \left(1 + \frac{p(1+\I{\ell\vert (j+1)})}{n_{j}}\right)\\
&\le~\prod_{r=m'}^{t'} \left(1 + \frac{2p}{b+w+ r(\ell+1)}\right)
\prod_{s=m'}^{t'} \prod_{j=1}^{\ell-1} \left( 1+ \frac{p}{b+w+s(\ell+1)+j}\right) \\
&\le~\prod_{r=m'}^{t'} \left(1 + \frac{2p}{b+w+ r(\ell+1)}\right)
\prod_{s=m'}^{t'} \left( 1+ \frac{p}{b+w+s(\ell+1)+1}\right)^{\ell-1} \\
&\le~ \exp\left(2p \sum_{r=m'}^{t'} \frac{1}{b+w+r(\ell+1)}\right)
\cdot \exp\left(p(\ell-1) \sum_{s=m'}^{t'} \frac{1}{b+w+s(\ell+1)+1}\right) \\
&\leq~\exp\left(p(\ell+1) \sum_{r=m'}^{t'} \frac{1}{b+w+r(\ell+1)}\right).
\end{align*}
Since $\int_0^a \frac{dx}{\gamma+\beta x} = \frac{1}{\beta} \log\left(1+\frac{a\beta}{\gamma}\right)$, $t'\le \frac{t}{\ell}$, 
and (using $b+w\geq 1$) there is a constant $f_{\ell}'$ not depending on $b,w,m,t$ such that
\be{
\frac{\alpha(b+w) +m}{\alpha(b+w) + \ell m'}\leq f_{\ell}',
}
the inequalities above give, for some positive constant $ f_{p,\ell}'$,
\ben{
\gamma_{m,t-1} \le f_{p,\ell}' \left( \frac{b+w+t'(\ell+1)}{b+w+m'(\ell+1)}\right)^{p}\le f_{p,\ell}' \left(\frac{\alpha(b+w)+t}{\alpha(b+w)+m}\right)^{p}. \label{in:gam}
}
Now, since $n_i = b+w + i + \lfloor i/\ell\rfloor$ and
\[
d_{i,p-1}=f_{p-1,\ell} w^{p-1}\cdot \left(1+\frac{i}{\alpha(b+w)}\right)^{p-1} \cdot \log\left(1+\frac{i}{\alpha(b+w)}\right),
\]
we have
\begin{align}
&\sum_{i=0}^{t-1} \frac{d_{i,p-1}e_{p}}{n_i}\gamma_{i+1,t-1} \notag \\
&\qquad\leq~\left(f_{p-1,\ell}  e_p f_{p,\ell}'\right) \cdot w^p \cdot \left(1+\frac{t}{\alpha(b+w)}\right)^p \cdot 
\sum_{i=0}^{t-1}\left(\frac{\alpha(b+w)}{\alpha(b+w)+i}\right)\frac{\log\left(1+\frac{i}{\alpha(b+w)}\right)}{n_i} \notag\\
&\qquad \leq \left(f_{p-1,\ell}  e_p f_{p,\ell}'\right) \cdot w^p \cdot \left(1+\frac{t}{\alpha(b+w)}\right)^p \cdot \log\left(1+\frac{t}{\alpha(b+w)}\right), \label{in:small}
\end{align}
where we have used again $\log(1+x)\leq x$ and that the sum is bounded by the appropriate integral.
Combining \eqref{in:p},~\eqref{in:gam} with $m=0$, and \eqref{in:small} implies that there is a constant $f_{p,\ell}$ such that
\[
\E{W_t^p} \le f_{p,\ell} \cdot w^p\cdot \left(1+ \frac{t}{\alpha(b+w)}\right)^{p} \cdot \left\{ 1+ \log \left(1+ \frac{t}{\alpha (b+w)}\right) \right\},
\]
and now (\ref{in:momentbd}) follows by induction. 
\end{proof}

\begin{lem}\label{lem:ucondbd}
Fix $b,w,p,i\in\N$. There is a constant $g_{p,\ell}$ independent of $b,w,t$ such that for $t\geq i\ell$
\[
\E{\cU_i(t;b,w)^p}\le g_{p,\ell} \left(\frac{\alpha(b+w)-1+t}{\alpha(b+w)-1+i\ell}\right)^{p+1}.
\]
\end{lem}

\begin{proof}
Write $\cM_j(\cdot) = \cM_j(\cdot;b,w)$, $\cU_j(\cdot) = \cU_j(\cdot;b,w)$, and $\cS_j=\cS_j(b,w)$, for all $j\in\N$.  It follows from~\eq{eq:ucondmdist} and~\eq{in:momentbd} that for all $t\geq1$ (note that the inequality trivially holds for $t\leq \cS_i$),
\ba{
\IE&\bkle{\cU_i(t)^p\big\vert \cM_i(t),\cS_i, \cM_i(\cS_i),\cU_0(\cS_i)} \\
&\qquad \le  f_{p,\ell} \left(1+\frac{\cM_i(t)-\cM_i(\cS_i)}{\alpha \cM_i(\cS_i)}\right)^{p}\left(1+\log\left(1+\frac{\cM_i(t)-\cM_i(\cS_i)}{\alpha \cM_i(\cS_i)}\right)\right) \\
&\qquad \leq  g_{p,\ell}' \left(\frac{\cM_i(t)}{\cM_i(\cS_i)}\right)^{p+1/2},
}
where the last inequality follows from $\alpha\leq 1/2$ and $\cM_i(t)\geq \cM_i(\cS_i)$,
$\log(x)\leq \sqrt{x}$ for $x\geq 1$, and $g_{p,\ell}'\geq 1$ (WLOG) is a constant.  
Averaging over $\cM_i(t)$ and using Jensen's inequality then yields
\[
\E{\cU_i(t)^{p}\big\vert \cS_i,\cM_i(\cS_i),\cU_0(\cS_i)} 
\le \frac{g'_{p,\ell}}{\cM_i(\cS_i)^{p+1/2}}\cdot \E{\cM_i(t)^{p+1}\Big\vert \cS_i, \cM_i(\cS_i),\cU_0(\cS_i)}^{\frac{p+1/2}{p+1}}.
\]
Furthermore, applying (\ref{eq:mdist}) and (\ref{in:momentbd}), we have for $t\geq \cS_i$,
\ba{
\IE&\bkle{\cM_i(t)^{p+1}\big\vert \cS_i, \cM_i(\cS_i),\cU_0(\cS_i)} \\
&\quad \le  f_{p+1,\ell}\cM_i(\cS_i)^{p+1}\cdot \left(1+\frac{t-\cS_i}{\alpha(\cU_0(\cS_i)+ \cM_i(\cS_i))}\right)^{p+1} \\	
&\hspace{2in} \times\left(1+\log\left(1+\frac{t-\cS_i}{\alpha(\cU_0(\cS_i)+ \cM_i(\cS_i))}\right)\right).
}
Now, since $\cS_i\ge i\ell$ and
$\cU_0(\cS_i)+\cM_i(\cS_i)=b+w+\cS_i + \lfloor \cS_i/\ell\rfloor$ is the total number of balls in the urn at time~$\cS_i$, we
have for $t\geq i\ell$ and some positive constant $g_{p,\ell}''$,
\[
\E{\cM_i(t)^{p+1}\Big\vert \cS_i, \cM_i(\cS_i),\cU_0(\cS_i)}
\le
g_{p,\ell}''\cM_i(\cS_i)^{p+1} \cdot \left(\frac{\alpha(b+w)-1+t}{\alpha(b+w)-1+i\ell}\right)^{p+3/2}.
\]
Altogether, 
\begin{align*}
&\E{\cU_i(t)^p\big\vert \cS_i, \cU_0(\cS_i),\cM_i(\cS_i)}\\
&\qquad\quad\le~ \frac{g_{p,\ell}}{\cM_i(\cS_i)^{p+1/2}} \cdot \left(\cM_i(\cS_i)^{p+1} \cdot \left(\frac{\alpha(b+w)-1+t}{\alpha(b+w)-1+i\ell}\right)^{p+3/2}\right)^{\frac{p+1/2}{p+1}}\\
&\qquad\quad\le~ g_{p,\ell}\left(\frac{\alpha(b+w)-1+t}{\alpha(b+w)-1+i\ell}\right)^{p+1}.
\end{align*}
Since the bound does not depend on $\cS_i$, $\cU_0(\cS_i)$, or $\cM_i(\cS_i)$, the lemma follows.
\end{proof}

\begin{proof}[{\bf Proof of \refL{lem:maxsize}}]
 Given $b,w,p\in\N$, it follows from a union bound, Markov's inequality, and \refL{lem:ucondbd} that   for any $x>0$,
\begin{align*}
&\p{\max_{i\in\N} \cU_i(t;b,w) \ge t/x}\le~  \frac{\sum_{i\in\N}\E{\cU_i(t;b,w)^p }}{t^p /x^{p}}\\
&\hspace{1in}\le~ g_{p,\ell}\frac{x^p}{t^p}
\sum_{1\le i\le \lfloor t/\ell\rfloor}\left(\frac{\alpha(b+w)-1+t}{\alpha(b+w)-1+i\ell}\right)^{p+1}\\
&\hspace{1in}\le~ g_{p,\ell}\frac{x^p}{t^p} 
\cdot \frac{(\alpha(b+w)-1+t)^{p+1}}{\ell}\int_0^{t} \frac{1}{(\alpha(b+w)-1+x)^{p+1}} dx\\
&\hspace{1in}\le ~ g_{p,\ell} \frac{x^p(\alpha(b+w)-1+t) ((\alpha(b+w)-1)t^{-1}+1)^p}{\ell p (\alpha(b+w))^p} \\
&\hspace{1in}\leq ~ g_{p,\ell}\frac{x^p(\alpha(b+w)-1+t) (t^{-1}+(\alpha(b+w))^{-1})^p}{\ell p }
\end{align*}
Next, note that for any $k,n\in\N$ with $n\ge k\ell$,
\[
S_{k,n} \coloneqq \max\left\{ |v(T)|: T\in \rT(n) -\rT_k(n)\right\}
\eqdist \max_{i\in\N}\cU_i\left(n - k\ell; (\ell+1)k,1\right).
\]
Then, for any $p\in\N$ and for sufficiently large $(n-k\ell)$, taking $t=n-k\ell$ and $x=k^{\frac{8}{3(\ell+1)}}$ in the above inequalities and recalling that $k=o(n)$ yields
\begin{align*}
\p{S_{k,n}\geq \frac{n-k\ell}{k^{\frac{8}{3(\ell+1)}}}}= O\left(k^{\frac{8p}{3(\ell+1)}-p} \cdot n  \right)
\end{align*}
Now using that $k(n)=\Omega\left(n^{1/100}\right)$, $\ell\geq 2$, and fixing $p\geq9000(\ell+1)$,
we have 
\[
\p{S_{k(n),n} \ge n\cdot k(n)^{-\frac{8}{3(\ell+1)}}} = O(n^{-2}).\qedhere
\] 
\end{proof}

\section{Proofs for generalized gamma distributions}\label{sec:moment}

\begin{proof}[{\bf Proof of \refL{lem:bd}}]
For all $\veps\in(0,3/7]$, $-\frac{\veps}{3}+\frac{\veps^2}{4}-\frac{\veps^3}{5}+\cdots< \frac{\veps}{3}+\frac{\veps^2}{4}+\frac{\veps^3}{5}+\cdots\le \frac{\veps}{3}+\frac{\veps^2}{3}+\cdots = \frac{\veps}{3(1-\veps)}\le \frac14$, so $(1-\veps)^{-k} = e^{-k\log(1-\veps)} = e^{k\left(\veps + \frac{\veps^2}{2} +\frac{\veps^3}{3}+\cdots\right)} \le e^{k\left(\veps + \frac{3\veps^2}{4}\right)}$, and $(1+\veps)^{-k} =e^{-k\log(1+\veps)}= e^{k\left(-\veps+\frac{\veps^2}{2}-\frac{\veps^3}{3}+\cdots \right)}< e^{k\left(-\veps+\frac{3\veps^2}{4}\right)}$. Next, let $E_1,E_2,\ldots$ be independent Exponential$(1)$-variables. For $k\in\N$, we may write $C_k = (E_1+\cdots+ E_k)^{\frac{1}{\ell+1}}$. By Markov's inequality and the previous derivation, for all $k\in\N$ and $\veps\in(0,3/7]$,
\begin{align*}
\p{C_k^{\ell+1}-k \ge \veps k}
\le&~ \frac{e^{-\veps k}\left(\E{e^{\veps E_1}}\right)^k}{e^{\veps^2 k}}
 = e^{-\veps k-\veps^2 k}\cdot (1-\veps)^{-k} \\
\le&~ e^{-\veps k- \veps^2 k}\cdot e^{\veps k + 3\veps^2 k/4}
 = e^{-\veps^2 k/4}.
\end{align*}
Similarly, for $\veps\in(0,3/7]$, $\p{k-C_k^{\ell+1}\ge \veps k} \le \frac{e^{\veps k}\left(\E{e^{-\veps E_1}}\right)^k}{e^{\veps^2k}}
\le e^{\veps k-\veps^2 k}\cdot (1+\veps)^{-k}
\le e^{\veps k - \veps^2 k} \cdot e^{-\veps k + 3\veps^2 k/4} = e^{-\veps^2 k/4}$. Now, take $\veps = k^{-1/4}$. For $k=\veps^{-4}\ge(7/3)^4$, the bounds above and the triangle inequality imply $\p{\left\vert C_k^{\ell+1} - k\right\vert \ge k^{3/4}}
\le 2e^{-\veps^2 k/4} = 2e^{-k^{1/2}/4}$.
\end{proof}

\begin{cor}\label{cor:bd}
For all integer $k\ge (7/3)^4$, with probability greater than $1-2 e^{-k^{1/2}/4}$,
\[
 k^{-\frac{1}{\ell+1}} \cdot \left(1-\frac{2}{\ell+1}\cdot  k^{-\frac14}\right)< \frac{1}{C_k} 
 < k^{-\frac{1}{\ell+1}} \cdot \left(1+\frac{2}{\ell+1}\cdot k^{-\frac14}\right).
\]
\end{cor}

\begin{proof}
It follows from \refL{lem:bd} that, for large enough $k\in\N$, $\p{\left\vert C_k^{\ell+1} - k\right\vert \ge k^{3/4}}
\le 2e^{-k^{1/2}/4}$. So with probability greater than $1-2 e^{-k^{1/2}/4}$,
\[
k^{-\frac{1}{\ell+1}}\cdot \left(1+k^{-1/4}\right)^{-\frac{1}{\ell+1}}
< \frac{1}{C_k} < k^{-\frac{1}{\ell+1}}\cdot \left(1-k^{-1/4}\right)^{-\frac{1}{\ell+1}}.
\]
 Taylor expansion then yields that 
\[
\left(1-k^{-1/4}\right)^{-\frac{1}{\ell+1}}=1+\frac{1}{\ell+1} k^{-1/4} + \frac{1}{\ell+1}\left(1+\frac{1}{\ell+1}\right) \frac{k^{-1/2}}{2!}+\cdots.
\]
For the rest of the proof we assume that $k>(3/2)^4$, then $(1+\frac{1}{\ell+1}) k^{-1/4}<1$. It follows that $\left(1-k^{-1/4}\right)^{-\frac{1}{\ell+1}} < 1+ \frac{2}{\ell+1}\cdot k^{-1/4}$. Similarly, $\left(1+k^{-1/4}\right)^{-\frac{1}{\ell+1}} >1-\frac{2}{\ell+1}\cdot k^{-1/4}$. 
\end{proof}

\begin{proof}
[{\bf Proof of \refC{cor:bd2}}]
Without loss of generality, assume that $\ell$ divides $n$. Since $k(n)^\alpha =o(n^{\alpha-1/4})$, it follows from \refL{lem:bd} and \refC{cor:bd} that, with probability greater than $ 1-2\sum\limits_{m=k(n)}^{n/\ell} e^{-m^{1/2}/4}$, for sufficiently large $n$ we simultaneously have $\left\vert C_{k(n)}-k(n)^{\frac{1}{\ell+1}}\right\vert < 10 k(n)^{\frac{1}{\ell+1}-\frac14}$ and
\begin{align*}
\sum_{m=k(n)}^{n/\ell} \frac{1}{C_m} >&~ \sum_{m=k(n)}^{n/\ell} m^{-\frac{1}{\ell+1}}
- \frac{2}{\ell+1}\sum_{m=k(n)}^{n/\ell} m^{-\frac{1}{\ell+1}-1/4}\notag\\
\ge&~ \frac{1}{\alpha}\cdot\left\{ \left(\frac{n}{\ell}\right)^{\alpha} - k(n)^{\alpha}\right\}- \frac{2}{\ell+1}\cdot \frac{1}{\alpha-1/4}\cdot \left\{\left(\frac{n}{\ell}\right)^{\alpha-1/4} - k(n)^{\alpha-1/4}\right\}\notag\\
\ge&~ \frac{1}{\alpha}\cdot \left(\frac{n}{\ell}\right)^{\alpha} - \frac{5}{n^{1/4}}\cdot \left(\frac{n}{\ell}\right)^{\alpha},
\end{align*}
and $\sum_{m=k(n)}^{n/\ell} \frac{1}{C_m}< \frac{1}{\alpha}\cdot \left(\frac{n}{\ell}\right)^\alpha + \frac{5}{n^{1/4}} \cdot \left(\frac{n}{\ell}\right)^\alpha$.
\end{proof}

\section*{Acknowledgement}
\addtocontents{toc}{\SkipTocEntry}

We thank Adrian R\"{o}llin for suggesting inserting random vertices into the real trees.

\end{document}